\documentclass[11pt,reqno]{amsart} 
\usepackage[applemac]{inputenc}
\usepackage{amsmath} 
\usepackage{amssymb} 
\usepackage{euscript}
\usepackage{mathabx}

\usepackage[colorlinks=true,linktocpage=true,pagebackref=false, urlcolor=black, citecolor=black,linkcolor=black]{hyperref}

\newtheorem{thm}{Theorem}[section]
\newtheorem{lem}[thm]{Lemma}
\newtheorem{prop}[thm]{Proposition}

\newtheorem{deflem}[thm]{Definition and Lemma}
\newtheorem{cor}[thm]{Corollary}

\newtheorem{mthm}{Theorem}

\theoremstyle{definition}
\newtheorem{defn}[thm]{Definition}

\theoremstyle{remark}
\newtheorem{remark}[thm]{Remark}
\newtheorem{remarks}[thm]{Remarks}
\newtheorem{example}[thm]{Example}
\newtheorem{examples}[thm]{Examples}

{\em}

{\em}

\newcounter{substep}
\def\thesubstep{\arabic{substep}}

\newenvironment{substeps}[1]{%
\refstepcounter{substep}\noindent{(\ref{#1}.\thesubstep)\ }\ }%
{\em}

\newcounter{subsubstep}
\def\thesubsubstep{\arabic{subsubstep}}

{\em}

\newcommand{\K}{{\mathbb K}} \newcommand{\N}{{\mathbb N}}
\newcommand{\Z}{{\mathbb Z}} \newcommand{\R}{{\mathbb R}}
 \newcommand{\C}{{\mathbb C}}

\newcommand{\ideal}{{\mathcal I}}
\newcommand{\ceros}{{\mathcal Z}}

\newcommand{\gtp}{{\mathfrak p}} \newcommand{\gtq}{{\mathfrak q}} 
 
\newcommand{\gta}{{\mathfrak a}} \newcommand{\gtb}{{\mathfrak b}}
\newcommand{\gtP}{{\mathfrak P}}

\newcommand{\gtg}{{\mathfrak g}}

\newcommand{\Fhaz}{{\EuScript F}}

\newcommand{\Jhaz}{{\EuScript I}}
\newcommand{\an}{{\EuScript O}}

\newcommand{\An}{{\EuScript A}}

\newcommand{\supp}{\operatorname{supp}}

\newcommand{\id}{\operatorname{id}}

\newcommand{\re}{\operatorname{Re}}
\newcommand{\ima}{\operatorname{Im}}

\newcommand{\x}{{\tt x}} \newcommand{\y}{{\tt y}} 
 
 \newcommand{\h}{{\tt h}}

\newcommand{\veps}{\varepsilon}
\newcommand{\ol }{\overline}
\newcommand{\qq}[1]{\langle{#1}\rangle}

\numberwithin{equation}{section}

\renewcommand{\theequation}{\thesection.\arabic{equation}}

\begin{document}

\title[On the Nullstellens\"atze for Stein spaces and $C$-analytic sets]{On the Nullstellens\"atze for Stein spaces\\ and $C$-analytic sets}

\author{Francesca Acquistapace, Fabrizio Broglia}
\address{Dipartimento di Matematica, Universit\`a degli Studi di Pisa, Largo Bruno Pontecorvo, 5, 56127 Pisa, Italy}
\email{acquistf@dm.unipi.it, broglia@dm.unipi.it}

\author{Jos\'e F. Fernando}
\address{Departamento de \'Algebra, Facultad de Ciencias Matem\'aticas, Universidad Complutense de Madrid, 28040 MADRID (SPAIN)}
\email{josefer@mat.ucm.es}

\thanks{Authors supported by Spanish GAAR MTM2011-22435. First and second authors also supported by Italian GNSAGA of INdAM and MIUR. This article is the fruit of the close collaboration of the authors in the last ten years and has been performed in the course of several research stays of the first two authors in the Department of Algebra at the Universidad Complutense de Madrid and of the third author in the Department of Mathematics at the Universit\`a di Pisa.}

\date{January 1st, 2014}

\begin{abstract}
In this work we prove the real Nullstellensatz for the ring $\an(X)$ of analytic functions on a $C$-analytic set $X\subset\R^n$ in terms of the \em saturation \em of \L ojasiewicz's radical in $\an(X)$: \em The ideal $\ideal(\ceros(\gta))$ of the zero-set $\ceros(\gta)$ of an ideal $\gta$ of $\an(X)$ coincides with the saturation $\widetilde{\sqrt[\text{\L}]{\gta}}$ of \L ojasiewicz's radical $\sqrt[\text{\L}]{\gta}$\em. If $\ceros(\gta)$ has `good properties' concerning Hilbert's 17th Problem, then $\ideal(\ceros(\gta))=\widetilde{\sqrt[\mathsf{r}]{\gta}}$ where $\sqrt[\mathsf{r}]{\gta}$ stands for the \em real radical \em of $\gta$. The same holds if we replace $\sqrt[\mathsf{r}]{\gta}$ with the \em real-analytic radical \em $\sqrt[\mathsf{ra}]{\gta}$ of $\gta$, which is a natural generalisation of the real radical ideal in the $C$-analytic setting. We revisit the classical results concerning (Hilbert's) Nullstellensatz in the framework of (complex) Stein spaces. 

Let $\gta$ be a saturated ideal of $\an(\R^n)$ and $Y_{\R^n}$ the germ of the support of the coherent sheaf that extends $\gta\an_{\R^n}$ to a suitable complex open neighbourhood of $\R^n$. We study the relationship between a \em normal primary decomposition \em of $\gta$ and the decomposition of $Y_{\R^n}$ as the union of its irreducible components. If $\gta:=\gtp$ is prime, then $\ideal(\ceros(\gtp))=\gtp$ if and only if the (complex) dimension of $Y_{\R^n}$ coincides with the (real) dimension of $\ceros(\gtp)$.
\end{abstract}

\maketitle

\section*{Introduction}\label{s1}
\renewcommand{\theequation}{I.\arabic{equation}}

In this paper we characterise the ideals $\gta$ of the algebra $\an(X)$ that have the zero property where $X$ is either a Stein space or a $C$-analytic set (also known as $C$-analytic subset of $\R^n$). Recall that an ideal $\gta$ of $\an(X)$ has the {\em zero property} if it coincides with the ideal $\ideal(\ceros(\gta))$ of all analytic functions on $X$ vanishing on its zero-set $\ceros(\gta)$. More generally, we approach the problem of determining the ideal $\ideal(\ceros(\gta))$ algebraically from an ideal $\gta$ of $\an(X)$. These problems are commonly known as Nullstellens\"atze. The complex and the real-analytic case have deserved the attention of specialists in both matters for a long time.

Our results are new for the general real case. Until now, all known results exist only for two particular situations: 
\begin{itemize}
\item[(1)] compact analytic spaces \cite{j2,rz1} or 
\item[(2)] analytic spaces of low dimensions \cite{ad0,bp}: $0,1$ or $2$. 
\end{itemize}
For the complex case we extend the classical Forster's Nullstellensatz by removing the condition that the involved ideal $\gta$, for which one computes $\ideal(\ceros(\gta))$, is closed.

\renewcommand{\thethm}{\arabic{thm}}
\subsection*{The complex case.}

The main known results concerning the complex analytic Nullstellensatz go back to the 1960's and are due to Forster \cite{of} and Siu \cite{s}. To state the main results, we fix a Stein algebra $\an(X):=H^0(X,\an_X)$, that is, the algebra of global analytic sections on a (reduced) Stein space $(X,\an_X)$. There are crucial differences concerning the behaviour of polynomial functions on an algebraic variety and analytic functions on a Stein space. Besides that $\an(X)$ is neither noetherian nor a unique factorization domain, two main obstructions appear to get a Nullstellensatz. The first one arises because there are proper prime ideals with empty zero-set while the second one appears because the `multiplicity' of an analytic function $G\in\an(X)$ vanishing (identically) on a discrete set can be unbounded. Thus, if another analytic function $F\in\an(X)$ vanishes on the zero-set of $G$ with multiplicity $1$, no power of $F$ can belong to the ideal $G\an(X)$. Classical examples of the previous situations, for which $\K$ denotes either $\R$ or $\C$, are the following:

\begin{example}\label{one}
Let ${\mathfrak U}$ be an ultrafilter of subsets of $\N$ containing all cofinite subsets. For an analytic function $F\in\an(\K)$ we denote the {\em multiplicity} of $F$ at the point $z\in\K$ with ${\rm mult}_z(F)$. Put $M(F,m):=\{\ell\in\N:\ {\rm mult}_\ell(F) \geq m\}$. Consider the non-empty set
$$
\gtp:=\{F\in\an(\K):\ M(F,m)\in{\mathfrak U}\ \forall m\geq0 \}.
$$

Let us check that $\gtp$ is a prime ideal. Indeed, let $F,G\in\gtp$. Then $M(F,m)\cap M(G,m)\subset M(F+G,m)$ because ${\rm mult}_\ell(F+G)\geq\min\{{\rm mult}_\ell(F),{\rm mult}_\ell(G)\}$, so $M(F+G,m)\in{\mathfrak U}$ for all $m\geq0$. On the other hand, if $F\in\gtp$ and $G\in\an(K)$, then ${\rm mult}_\ell(FG)={\rm mult}_\ell(F)+{\rm mult}_\ell(G)$, so $M(FG,m)\supset M(F,m)\in{\mathfrak U}$ for all $m\geq0$.

Suppose $F_1F_2\in\gtp$ but $F_1,F_2\not\in\gtp$. Then there exist $m_1,m_2\geq0$ such that 
$$
M(F_1,m_1),M(F_2,m_2)\notin{\mathfrak U}.
$$ 
Take $m_0:=\max\{m_1,m_2\}$ and note $M(F_1,m_0),M(F_2,m_0)\notin{\mathfrak U}$; hence, $M(F_1,m_0)\cup M(F_2,m_0)\not\in{\mathfrak U}$. On the other hand, 
$$
M(F_1,m_0)\cup M(F_2,m_0)\supset M(F_1F_2,2m_0)\in{\mathfrak U}, 
$$
so also $M(F_1,m_0)\cup M(F_2,m_0)\in{\mathfrak U}$, which is a contradiction. Thus, $\gtp$ is a prime ideal.

Finally, observe $\ceros(\gtp)=\varnothing$. For each $k\geq1$ let $G_k\in\an(\K)$ be an analytic function such that $\ceros(G_k)=\{\ell\in\N:\ \ell\geq k\}$ and ${\rm mult}_\ell(G_k)=\ell$ for all $\ell\geq k$. Since ${\mathfrak U}$ contains all cofinite subsets, we deduce that each $G_k\in\gtp$, so $\ceros(\gtp)\subset\bigcap_{k\geq1}\ceros(G_k)=\varnothing$.
\end{example}

\begin{example}\label{two}
Let $F,G\in\an(\K)$ be the analytic functions given by the infinite products: 
$$
F(z):=\prod_{n\geq1}\Big(1-\frac{z}{n^2}\Big)\quad\text{ and }\quad G(z):=\prod_{n\geq1}\Big(1-\frac{z}{n^2}\Big)^n
$$ 
for all $z\in\K$. The zero-sets of $F$ and $G$ coincide with the set $\{n^2:\ n\geq1\}$ and we denote $\gta=G\an(\K)$. If the classical Nullstellensatz held for $\an(\K)$, there would exist an integer $m\geq 0$ and an analytic function $H\in\an(\K)$ such that $F^m=GH$. Let us compare multiplicities in the previous formula at the point $(m+1)^2$: the left hand side vanishes at the point $(m+1)^2$ with multiplicity $m$ while the right hand side vanishes at the point $(m+1)^2$ with multiplicity $\geq m+1$, which is a contradiction. Thus, we conclude $\ideal(\ceros(\gta))\neq\sqrt{\gta}$.
\end{example}

To control these difficulties, Forster showed first that the prime \em closed \em ideals $\gtp$ of $\an(X)$ endowed with its usual Fr\'{e}chet's topology \cite[VIII.A]{gr} have the zero property, that is, $\ideal(\ceros(\gtp))=\gtp$. Afterwards he proved that the closed ideals $\gta$ of $\an(X)$ admit (as in the noetherian case) a \em normal primary decomposition \em (see \S\ref{npd}). Of course, for a general normal primary decomposition there exist countably many primary ideals $\gtq_i$. 

In this context we extend Forster's Nullstellensatz (see Section \ref{s3} for precise statements) to the non-closed case as we state in the next result. Given an ideal $\gtb$ of $\an(X)$, we denote its closure with respect to the usual Fr\'{e}chet's topology of $\an(X)$ with $\ol{\gtb}$. 

\begin{mthm}[Nullstellensatz]\label{perfect}
Let $(X,\an_X)$ be a Stein space and $\gta\subset\an(X)$ an ideal. Then 
$$
\ideal(\ceros(\gta))=\overline{\sqrt{\gta}}. 
$$
In particular, $\ideal(\ceros(\gta))=\gta$ if and only if ${\gta}$ is radical and closed.
\end{mthm}

We will show in Remark \ref{int} that if $\gta$ is a closed ideal with normal primary decomposition $\gta=\bigcap_{i\in I}\gtq_i$, we have $\overline{\sqrt{\gta}}=\bigcap_{i\in I}\sqrt{\gtq_i}$. Nevertheless, it may happen that $\overline{\sqrt{\gta}}\supsetneq\sqrt{\gta}$ because the radical of a closed ideal $\gta$ needs not to be closed (see Section \ref{s3}). However, the radical of a closed primary ideal $\gtq$ is still closed, see Lemma \ref{lemmaprimary}.

\subsection*{The real case.}

The situation in the real case is more demanding. We have similar initial difficulties to the ones described in the complex analytic case. Examples \ref{one} and \ref{two} are generalised to the real case as follows. 

\begin{examples}\label{three}
(i) The ideal ${\gta}$ in Example \ref{one} is a {\em real} ideal, that is, if a sum of squares $\sum_{i=1}^pf_i^2$ in $\an(\R)$ belongs to $\gta$, then each $f_i\in\gta$. Indeed, assume $f:=\sum_{i=1}^pf_i^2\in\gta$. Since
$$
{\rm mult}_\ell(f)=2\min\{{\rm mult}_\ell(f_1),\ldots,{\rm mult}_\ell(f_p)\},
$$ 
we deduce $M(f,2m)\subset M(f_i,m)$ for all $m\geq0$ and $i=1,\ldots,p$. Thus, since each $M(f,2m)\in{\mathfrak U}$, we deduce $M(f_i,m)\in{\mathfrak U}$ for all $m\geq0$, that is, each $f_i\in\gta$. Thus, $\gta$ is a real prime ideal with empty zero-set.

(ii) Concerning Example \ref{two}, let $f,g\in\an(\R)$ be the corresponding analytic functions defined by the formulas proposed there and let $\gta:=g\an(\R)$. We want to show $\ideal(\ceros(\gta))\neq\sqrt[\mathsf{r}]{\gta}$ where $\sqrt[\mathsf{r}]{\gta}$ is the real radical of $\gta$ (see equation \eqref{eq1} below). Indeed, let us prove $f\not\in\sqrt[\mathsf{r}]{\gta}$. Otherwise there would exist an integer $m\geq1$ and analytic functions $h_1,\ldots,h_p,h\in\an(\R)$ such that
$$
f^{2m}+\sum_{i=1}^ph_i^2=gh.
$$
Comparing the orders at both sides of the previous equality at the point $(2m+1)^2$, we obtain a contradiction.
\end{examples}

Consider a $C$-analytic subset $X\subset \R^n$ and let $\ideal(X)$ be the ideal of all (real) analytic functions vanishing on $X$. The structure sheaf of $X$ is the coherent sheaf $\an_X:=\an_{\R^n}/\ideal(X)\an_{\R^n}$. Its ring of global analytic sections 
$$
\an(X):=H^0(X,\an_X)=\an(\R^n)/\ideal(X)
$$ 
can be seen as a subset of the Stein algebra $\an(\widetilde{X})$ of its \em complexification \em $\widetilde{X}$ (understood as a complex analytic set germ at $X$, see \S\ref{below1}). We stress that $X$ needs not to be coherent as an analytic set. Recall also that Cartan proved in \cite[VIII.Thm.4, pag.60]{c1} that if $Y$ is a Stein space, the closure of an ideal $\gtb$ of $\an(Y)$ coincides with its \em saturation\em
$$
\tilde{\gtb}:=\{F\in\an(Y):\ F_z\in\gtb\an_{Y,z}\ \forall\,z\in Y\}=H^0(Y,\gtb\an_Y).
$$
We endow $\an(X)$ with the topology induced by Fr\'{e}chet's topology of $\an(\widetilde{X})$ but now the saturation
$$
\tilde{\gta}:=\{f\in\an(X):\ f_x\in\gta\an_{X,x}\ \forall\,x\in X\}=H^0(X,\gta\an_X)
$$
of an ideal $\gta$ of $\an(X)$ does not need to be closed. However, as de Bartolomeis proved in \cite{db,db1}, each \em saturated ideal $\gta$ of $\an(X)$ \em (that is, such that $\gta=\tilde{\gta}$) admits a \em normal primary decomposition \em similar to the one devised by Forster in the complex case. Note also that the previous definition of saturation coincides with the one proposed by Whitney for ideals in the ring of smooth functions over a real smooth manifold \cite[II.1.3]{m}.

Before stating our main result, we introduce some terminology. Given $f,g\in\an(X)$, we say that $f\geq g$ if $f(x)\geq g(x)$ for all $x\in X$. Given an ideal $\gta$ of $\an(X)$, we define its {\em \L ojasiewicz radical} as
\begin{equation}\label{eq0}
\sqrt[\text{\L}]{\gta}:=\{g\in\an(X):\ \exists\,f\in\gta,\ m\geq1\text{ such that } f-g^{2m}\geq0\}.
\end{equation}

The notion of \L ojasiewicz radical has been used by many authors to approach different problems mainly related to rings of germs, see for instance \cite{n}, \cite[p. 104]{d}, \cite[1.21]{k} or \cite[\S6]{dm} but also in the global smooth case \cite{abn}. We say that an ideal $\gta$ of $\an(X)$ is \em convex \em if each $g\in\an(X)$ satisfying $|g|\leq f$ for some $f\in\gta$ belongs to $\gta$. In particular, \L ojasiewicz's radical $\sqrt[\text{\L}]{\gta}$ of an ideal $\gta$ of $\an(X)$ is a radical convex ideal. Our main result in this setting is the following. 

\begin{mthm}[Real Nullstellensatz]\label{perfect2}
Let $X\subset\R^n$ be a $C$-analytic set and $\gta$ an ideal of the ring $\an(X)$. Then 
$$
\ideal(\ceros(\gta))=\widetilde{\sqrt[\text{\L}]{\gta}}.
$$
In particular, $\ideal(\ceros(\gta))=\gta$ if and only if $\gta$ is a convex, radical and saturated ideal.
\end{mthm} 

If we compare the previous result to the real Nullstellensatz for the ring of polynomial functions on a real algebraic variety, we observe that \L ojasiewicz's radical plays an analogous role to the one performed by the classical {\em real radical}. In our context the real radical of an ideal $\gta$ of $\an(X)$ is
\begin{equation}\label{eq1}
\sqrt[\mathsf{r}]{\gta}:=\Big\{f\in\an(X):\ f^{2m}+\sum_{k=1}^pa_k^2\in\gta\ \text{and}\ a_i\in\an(X),\ m,p\geq0\Big\}.
\end{equation}
Recall that $\gta$ is a \em real ideal \em if $\gta=\sqrt[\mathsf{r}]{\gta}$.

It is natural to search relations between both radicals. This question forces us to compare positive semidefinite analytic functions with sums of squares of analytic functions in close relation to Hilbert's 17th Problem for the analytic setting \cite{abfr3}. In Section \ref{s4} we see that both radicals coincide in the abstract setting of the real spectrum of a ring $A$. In Section \ref{s6} we prove the equality $\widetilde{\sqrt[\text{\L}]{\gta}}=\widetilde{\sqrt[\mathsf{r}]{\gta}}$ for an ideal $\gta$ of $\an(X)$ with the property that every positive semidefinite analytic function whose zero-set is $Z:=\ceros(\gta)$ can be represented as a (finite) sum of squares of meromorphic functions on $X$. Any $C$-analytic set $Z\subset\R^n$ with the previous property will be called $H$-{\em set}. Some examples of $H$-sets are the following: discrete sets \cite{bks} and compact sets \cite{j2,rz1}. Moreover, if $X$ is either an analytic curve or a coherent analytic surface, every analytic subset of $X$ is an $H$-set \cite{abfr1,abfr2}.

Since infinite (convergent) sums of squares of meromorphic functions make sense in $\an(X)$ (see Section \ref{s2} and \cite{abf,abfr3,bp}), we define the {\em real-analytic radical} of an ideal $\gta$ of $\an(X)$ as 
\begin{equation}\label{eq2}
\sqrt[\mathsf{ra}]{\gta}:=\Big\{f\in\an(\R^n):\ f^{2m}+\sum_{k\geq1}a_k^2\in\gta\ \text{and}\ a_i\in\an(\R^n),\ m\geq0\Big\}.
\end{equation}
We say that $\gta$ is a \em real-analytic ideal \em if $\gta=\sqrt[\mathsf{r}]{\gta}$.

The equality $\widetilde{\sqrt[\mathsf{ra}]{\gta}}=\widetilde{\sqrt[\text{\L}]{\gta}}$ holds for an ideal $\gta$ of $\an(X)$ with the property that every positive semidefinite analytic function whose zero-set is $Z:=\ceros(\gta)$ can be represented as an infinite sum of squares of meromorphic functions on $X$. We call those $C$-analytic sets with the previous property $H^{\mathsf a}$-{\em sets}. An example of a $H^\mathsf{a}$-set is a locally finite union of disjoint compact analytic sets. Thus, if all connected components of $X$ are compact, then all $C$-analytic subsets of $X$ are $H^\mathsf{a}$-sets. The following result collects all this information.

\begin{mthm}\label{h17nss}
Let $X\subset\R^n$ be a $C$-analytic set and $\gta$ an ideal of $\an(X)$ such that $\ceros(\gta)$ is a $H$-{set}. Then 
$$
\ideal(\ceros(\gta))=\widetilde{\sqrt[\mathsf{r}]{\gta}}. 
$$
In particular, $\ideal(\ceros(\gta))=\gta$ if and only if $\gta$ is real and saturated. The previous statements hold for a $H^{\mathsf a}$-{set} $\ceros(\gta)$ replacing `real' by `real-analytic'.
\end{mthm}

The previous result applies if $X$ is either an analytic curve, a coherent analytic surface or a $C$-analytic set whose connected components are all compact, so the real Nullstellensatz holds for such an $X$ in terms of the real radical (or the real-analytic radical). 

In Section \ref{s7} we prove that the class of ideals of $\an(X)$ that have the zero property enjoys the expected properties as it happens with the corresponding class in the algebraic setting. More precisely, it holds:

\begin{mthm}\label{complexreal}
Let $\gtq\subset\an(\R^n)$ be a saturated primary ideal. Then the following assertions are equivalent
\begin{itemize}
\item[(i)] $\ideal(\ceros(\gtq))=\sqrt{\gtq}$.
\item[(ii)] $\dim(\ceros_\C(\gtq))=\dim(\ceros(\gtq))$
\item[(iii)] There exists $x\in\ceros(\gtq)$ such that $\dim(\ceros(\gtq\an_{\R^n,x}))=\dim(\ceros(\gtq\an_{\C^n,x}))$.
\end{itemize}
\end{mthm}

As it is well-known, condition (iii) in Theorem \ref{complexreal} is equivalent to the existence of a regular point $y\in\ceros(\gtq)$ for the ideal $\sqrt{\gtq}$. Recall that $y\in\ceros(\gtq)$ is {\em regular} for the ideal $\sqrt{\gtq}$ if $\dim(\ceros(\gtq)_y)=k$ and there exists $f_{k+1},\ldots,f_n\in\sqrt{\gtq}$ such that ${\rm rk}(\nabla f_{k+1}(y),\ldots,\nabla f_n(y))=n-k$. The two previous conditions imply that $\ceros(\gtq)\cap U=\ceros(f_{k+1},\ldots,f_n)\cap U$ in a neighbourhood $U$ of $x$.

\subsection*{Structure of the article} In Section \ref{s2} we state Forster's and de Bartolomeis' normal primary decompositions for saturated ideals and recall the meaning of infinite sums of squares in the real-analytic setting. Section \ref{s3} is devoted to the complex Nullstellensatz while the real Nullstellensatz is the content of Section \ref{s5}. In Section \ref{s4} we see that \L oja\-sie\-wicz's radical and the real radical coincide in the abstract setting. We prove in Section \ref{s6} that an affirmative answer for Hilbert's $17^{th}$ Problem implies that the saturations of \L oja\-sie\-wicz's radical and the real radical coincide. We also discuss certain properties concerning {\em convex} and {\em quasi-real} ideals. Finally, we analyse the geometric meaning of the real Nullstellensatz for the ideal $\gta$ in Section \ref{s7}. To that end, we compare the real dimension of the $C$-analytic set $Z:=\ceros(\gta)$ and the complex dimension of the germ $\ceros(\gta\otimes\C)$.

\section{Preliminaries on analytic geometry and saturated ideals}\label{s2}
\renewcommand{\thethm}{\thesection.\arabic{thm}}
\renewcommand{\theequation}{\thesection.\arabic{equation}}

Although we deal with real-analytic functions, we make an extended use of complex analysis. In the following \em holomorphic functions \em refer to the complex case and \em analytic functions \em to the real case. For a further reading about holomorphic functions we refer the reader to \cite{gr}. 

\subsection{General terminology.}\label{below1}

Denote the coordinates in $\C^n$ with $z:=(z_1,\ldots,z_n)$ where $z_i:=x_i+\sqrt{-1}y_i$. As usual $x_i:=\re(z_i)$ and $y_i:=\ima(z_i)$ are respectively the \em real \em and the \em imaginary parts \em of $z_i$. Consider the conjugation $\sigma:\C^n\to\C^n,\ z\mapsto\ol{z}:=(\ol{z_1},\ldots,\ol{z_n})$ of $\C^n$, whose set of fixed points is $\R^n$. A subset $A\subset\C^n$ is \em invariant \em if $\sigma(A)=A$. Obviously, $A\cap\sigma(A)$ is the biggest invariant subset of $A$. Let $\Omega\subset\C^n$ be an invariant open set and $F:\Omega\to\C$ a holomorphic function. We say that $F$ is \em invariant \em if $F(z)=\ol{(F\circ\sigma)(z)}$ for all $z\in\Omega$. This implies that $F$ restricts to a (real) analytic function on $\Omega\cap\R^n$. Conversely, if $f$ is analytic on $\R^n$, it can be extended to an invariant holomorphic function $F$ on some invariant open neighbourhood $\Omega$ of $\R^n$. In general,
$$
\Re(F):\Omega\to\C,\ z\mapsto\tfrac{F(z)+\ol{(F\circ\sigma)(z)}}{2}\quad\text{and}\quad\Im(F):\Omega\to\C,\ z\mapsto\tfrac{F(z)-\ol{(F\circ\sigma)(z)}}{2\sqrt{-1}}
$$ 
are the \em real \em and the \em imaginary \em parts of $F$, which satisfy $F=\Re(F)+\sqrt{-1}\,\Im(F)$. An analytic subsheaf $\Fhaz$ of $\an_{\Omega}$ is called \em invariant \em if for each open invariant subset $U\subset\Omega$ and each $F\in H^0(U,\Fhaz)$, the holomorphic function $\ol{F\circ\sigma}\in\Fhaz(U)$. If $\Fhaz$ is an invariant sheaf on $\Omega$ and $F_1,\ldots,F_r\in H^0(\Omega,\Fhaz)$ generate $\Fhaz_z$ as an $\an_{\Omega,z}$-module for some $z\in\Omega$, then also $\Re(F_1),\Im(F_1),\ldots,\Re(F_r),\Im(F_r)$ generate $\Fhaz_z$ as an $\an_{\Omega,z}$-module.

We will use $\ceros(\cdot)$ to denote the zero-set of $(\cdot)$ and $\ideal(\cdot)$ to denote the ideal of functions vanishing identically on $(\cdot)$. For instance, if $(X,\an_X)$ is either a Stein space or a real coherent analytic space and $S\subset\an(X)$, \em the zero-set of $S$ \em is 
$$
\ceros(S):=\{x\in X:\ F(x)=0\ \forall F\in S\}.
$$ 
If $Z\subset X$, the \em ideal of $Z$ \em is 
$$
\ideal(Z):=\{F\in\an(X):\ F(x)=0\ \forall x\in Z\}. 
$$
For the sake of clearness we denote the elements of $\an(X)$ with capital letters if $(X,\an_X)$ is a Stein space and with small letters if $(X,\an_X)$ is a real coherent analytic space. If a property holds for both types of spaces, we keep capital letters.

If $(X,\an_X)$ is a coherent (paracompact) real-analytic space, there exists a (paracompact) complex analytic space $(\widetilde{X},\an_{\widetilde{X}})$ such that
\begin{itemize}
\item[(i)] $X\subset\widetilde{X}$ is a closed subset and $\an_{\widetilde{X},x}=\an_{X,x}\otimes\C$ for all $x\in X$.
\item[(ii)] There exists an antiholomorphic involution $\sigma:\widetilde{X}\to\widetilde{X}$ whose fixed locus is $X$.
\item[(iii)] $X$ has a fundamental system of invariant open Stein neighbourhoods in $\widetilde{X}$. 
\item[(iv)] If $X$ is reduced, then $\widetilde{X}$ is also reduced.
\end{itemize} 
The analytic space $(\widetilde{X},\an_{\widetilde{X}})$ is called a \em complexification \em of $X$. It holds that the germ of $(\widetilde{X},\an_{\widetilde{X}})$ at $X$ is unique up to an isomorphism. For further details see \cite{c2,t,wb}.

A $C$-analytic set $X\subset\R^n$ endowed with its (coherent) structure sheaf $\an_X=\an_{\R^n}/\ideal(X)\an_{\R^n}$ has a well-defined complexification exactly as above, except for the second condition in (i), which may fail for the points of a $C$-analytic subset $Y\subset X$ of smaller dimension. From now on a (reduced) real-analytic space is a pair $(X,\an_X)$ constituted of a $C$-analytic set $X\subset\R^n$ and its structure sheaf $\an_X$.
 
\subsection{Saturated and closed ideals.}
Let $(X,\an_X)$ be either a Stein space or a real-analytic space and $\gta\subset\an(X)$ an ideal.
We consider its \em saturation \em 
$$ 
\tilde{\gta}:=\{F\in\an(X):\ F_x\in\gta\an_{X,x}\ \forall\,x\in X\}.
$$
Of course, the ideal $\gta$ is \em saturated \em if $\tilde{\gta}=\gta$.

In the complex case $\tilde{\gta}$ coincides with the closure of $\gta$ in $\an(X)$ endowed with its usual Fr\'{e}chet topology \cite[VIII.Thm.4, pag.60]{c1}. Thus, saturated ideals coincide with closed ideals. If $(X,\an_X)$ is a reduced Stein space, its Fr\'{e}chet topology is induced by a countable collection of the natural seminorms $\|\cdot\|_m:=\sup_{K_m}\{|\cdot|\}$ where $\{K_m\}_{m\geq1}$ is an exhaustion of $X$ by compact sets. Of course, this topology does not depend on the chosen exhaustion \cite[VIII.A]{gr}.

On the other hand, if $(X,\an_X)$ is a real-analytic space, the inherited topology on $\an(X)$ is induced by the following convergence: \em A sequence $\{f_k\}_{k\geq1}$ of elements of $\an(X)$ converges to $f\in\an(X)$ if there exist a complexification $(\widetilde{X},\an_{\widetilde{X}})$ of $(X,\an_X)$ and holomorphic extensions $F_k$ of $f_k$ and $F$ of $f$ such that $F_k$ converges to $F$ in $H^0(\widetilde{X},\an_{\widetilde{X}})$ endowed with its Fr\'{e}chet topology \em \cite[\S1.5]{db}. With this topology $\an(X)$ is a complete topological $\R$-algebra.

The saturation arises `naturally' when dealing with Nullstellens\"atze to manage the existence of proper prime ideals and proper real prime ideals with empty zero-set (see Examples \ref{one} and \ref{three} in the Introduction). 

\subsection{Closed primary ideals and normal primary decomposition.}\label{npd}

Let $(X,\an_X)$ be either a Stein space or a real-analytic space. One of the main properties of the closed and saturated ideals of $\an(X)$ is that they enjoy a locally finite primary decomposition. Before entering into further details, we recall some preliminary definitions. Given a collection of ideals $\{\gta_i\}_{i\in I}$ of $\an(X)$, we say that it is {\em locally finite} if the family of their zero-sets $\{\ceros(\gta_i)\}_{i\in I}$ is locally finite in $X$. A decomposition $\gta=\bigcap_{i \in I}\gta_i$ of an ideal $\gta$ of $\an(X)$ is called {\em irredundant} if $\gta\neq\bigcap_{i\in K}\gta_i$ for each proper subset $K\subsetneq I$. Moreover, a primary decomposition $\gta=\bigcap_{i\in I}\gtq_i$ of an ideal $\gta$ of $\an(X)$ is called {\em normal} if it is locally finite, irredundant and the associated prime ideals $\gtp_i:=\sqrt{\gtq_i}$ are pairwise distinct. As usual, a primary ideal $\gtq_j$ is called an {\em isolated primary component} if $\gtp_j$ is minimal among the primes $\{\gtp_i\}_{i\in I}$. Otherwise, $\gtq_j$ is an {\em immersed primary component}.

Before we present the normal primary decomposition of saturated ideals due to Forster, we recall some results concerning saturated primary ideals.
 
\begin{lem}\label{lemmaprimary}
Let $\gtq\subset\an (X)$ be a primary ideal and $F\in\an(X)$. We have: 
\begin{enumerate}
\item[(i)] If $x\in\ceros(\gtq)$, then $F\in\gtq$ if and only if $F_x\in\gtq\an_{X,x}$.
\item[(ii)] $\gtq$ is saturated if and only if $\ceros(\gtq)\neq\varnothing$.
\item[(iii)] $\ceros(\gtq)$ is connected. 
\end{enumerate}
\end{lem}
\begin{proof}
(i) See \cite[\S3.1.Lem.]{of} and \cite[2.1.2]{db}. In the statement of both results the authors assume that the ideal $\gtq$ is saturated but this fact is only used to assure $\ceros(\gtq)\neq\varnothing$.

(ii) The `only if' implication is clear. For the converse, choose a point $x\in\ceros(\gtq)$ and observe that by (i) $\gtq=\{F\in\an(X):\ F_x\in\gtq\an_{X,x}\}$; hence, $\gtq$ is the `saturation' of a local ideal, so it is saturated.

(iii) If $(X,\an_X)$ is a Stein space, the result follows from Theorem \ref{primary}. If $(X,\an_X)$ is a real-analytic space, we recall a classical trick. Assume by contradiction that $\ceros(\gtq)$ is not connected and let $Y_1,Y_2\subset\ceros(\gtq)$ be two closed disjoint subsets such that $\ceros(\gtq)=Y_1\cup Y_2$. Observe in particular that $\gtq$ must be saturated. Let $f\in\gtq$ be such that $\ceros(\gtq)=\ceros(f)$ (see Lemma \ref{crespina} below) and $g\in X$ an analytic function such that $g$ is strictly positive on $Y_1$ and strictly negative on $Y_2$ (use Whitney's approximation lemma to construct $g$). Observe that $\ceros(f^2+g^2)=\varnothing$, so $h_i=\sqrt{f^2+g^2}+(-1)^ig$ is an analytic function whose zero-set is $Y_i$. Moreover, $h_1h_2=f^2\in\gtq$. However, $h_1,h_2\not\in\sqrt{\gtq}$ because neither of them vanishes on $\ceros(\gtq)$, which is a contradiction. Hence, $\ceros(\gtq)$ is connected.
\end{proof}

\begin{lem}\em (\cite[\S4.Hilfssatz 5]{of} and \cite[2.2.10]{db1}) \em\label{db1210}
Let $\{\gta_i\}_{i\in I}$ be a locally finite family of saturated ideals of $\an(X)$ and $\gtp\subset\an(X)$ a prime saturated ideal such that $\bigcap_{i\in I}\gta_i\subset\gtp$. Then there exists $i\in I$ such that $\gta_i\subset\gtp$. 
\end{lem}

Now we recall the normal primary decomposition of saturated ideals of $\an(X)$.

\begin{prop}\em (\cite[\S5]{of} and \cite[Thm. 2.3.6]{db}) \em\label{propproprieta}
Let $\gta\subset\an(X)$ be a saturated ideal of $\an(X)$. Then $\gta$ admits a normal primary decomposition $\gta=\bigcap_i\gtq_i$ such that all primary ideals $\gtq_i$ are saturated. Moreover, the prime ideals $\gtp_i:=\sqrt{\gtq_i}$ and the primary isolated components are uniquely determined by $\gta$ and do not depend on the normal primary decomposition of $\gta$.
\end{prop}
\begin{remark}\label{int}
Let us briefly show that if $\gta\subset\an(X)$ is a saturated ideal of $\an(X)$ and $\gta=\bigcap_i\gtq_i$ is a normal primary decomposition (such that all primary ideals $\gtq_i$ are saturated), then
\begin{equation}\label{inteq}
\widetilde{\sqrt{\gta}}=\bigcap_{i\in I}\sqrt{\gtq_i}.
\end{equation}
Indeed, by Lemma \ref{lemmaprimary}(ii) each $\sqrt{\gtq_i}$ is a prime saturated ideal. On the other hand, for each $x\in X$ there exists a finite set $J_x\subset I$ such that
$$
\gtq_i\an_{X,x}=\an_{X,x}\quad\text{(and so}\quad\sqrt{\gtq_i}\an_{X,x}=\an_{X,x})\quad\forall i\not\in J_x. 
$$
Thus, we get
\begin{multline*}
\bigcap_{i\in I}(\sqrt{\gtq_i}\an_{X,x})=\bigcap_{i\in J_x}(\sqrt{\gtq_i}\an_{X,x})=\Big(\bigcap_{i\in J_x}\sqrt{\gtq_i}\Big)\an_{X,x}=\sqrt{\Big(\bigcap_{i\in J_x}\gtq_i\Big)}\an_{X,x}\\
=\sqrt{\Big(\bigcap_{i\in J_x}\gtq_i\Big)\an_{X,x}}=\sqrt{\Big(\bigcap_{i\in I}\gtq_i\Big)\an_{X,x}}=\sqrt{\gta\an_{X,x}}=\sqrt{\gta}\an_{X,x}.
\end{multline*}
Finally, by \cite[Proof of Theorem 2.2.2, p.48]{db} we conclude the required equality
$$
\widetilde{\sqrt{\gta}}=H^0(X,\sqrt{\gta}\an_{X,x})=H^0\Big(X,\bigcap_{i\in I}(\sqrt{\gtq_i}\an_{X,x})\Big)=\bigcap_{i\in I}\sqrt{\gtq_i}.
$$
\end{remark}

As the reader can straightforwardly check, the normal primary decompositions enjoy the good behaviour one can expect when dealing with radical, real and real-analytic ideals.

\begin{cor}\label{lemmaidreal}
Let $\gta\subset\an(X)$ be a saturated ideal and $\gta=\bigcap_i\gtq_i$ a normal primary decomposition of $\gta$. We have:
\begin{enumerate}
\item[(i)] If $\gta$ is a radical, then each $\gtq_i$ is prime and the normal primary decomposition is unique.
\item[(ii)] If $\gta$ is a real (resp. real-analytic) ideal, every $\gtq_i$ is a real (resp. real-analytic) prime ideal and the normal primary decomposition is unique.
\end{enumerate}
\end{cor}

\subsection{Infinite sum of squares.}

Let $(X,\an_X)$ be a real-analytic space. Following the propositions in \cite[1.3]{abfr3} for a real-analytic manifold, we say that an element $f\in\an(X)$ is \em an infinite sum of squares of meromorphic functions on $X$ \em if there exists a non-zero divisor $g\in\an(X)$ such that $g^2f$ is an absolutely convergent series $\sum_{k\geq1}f_k^2$ in $\an(X)$, that is, there exist a complexification $(\widetilde{X},\an_{\widetilde{X}})$ of $(X,\an_X)$ and holomorphic extensions $F_k$ of $f_k$, $F$ of $f$ and $G$ of $g$ such that $G^2F=\sum_{k\geq1}F_k^2$ and $\sum_{k\geq1}F_k^2$ is an absolutely convergent series with respect to the Fr\'{e}chet topology of $H^0(\widetilde{X},\an_{\widetilde{X}})$. In other words, for each compact set $K\subset\widetilde{X}$ the series $\sum_{k\geq1}\sup_K|F^2_k|$ is convergent. For further details see \cite{abf,abfr3,fe}.

\section{The complex analytic Hilbert's Nullstellensatz}\label{s3}

The purpose of this section is to prove Theorem \ref{perfect}. We recall Forster's results about the Nullstellensatz for Stein algebras when dealing with closed ideals.

\begin{thm}[Closed primary case]\label{primary}
Let $(X,\an_X)$ be a Stein space and $\gtq\subset\an(X)$ a closed primary ideal. Then
$$
\ideal(\ceros(\gtq))=\sqrt{\gtq}.
$$
Moreover,
\begin{itemize}
\item[(i)] There exists a positive integer $m\geq 1$ such that $(\sqrt{\gtq})^m\subset\gtq$.
\item[(ii)] If $\gtp\subset\an(X)$ is a closed prime ideal, then $\ideal(\ceros(\gtp))=\gtp$.
\end{itemize}
\end{thm}
\begin{thm}[Closed general case]\label{fo}
Let $(X,\an_X)$ be a Stein space and $\gta\subset\an(X)$ a closed ideal. Consider a normal primary decomposition $\gta=\bigcap_{i\in I}\gtq_i$ of $\gta$. For each $i\in I$, define
\begin{align*}
\h(\gtq_i,\gta):=&\inf\left\{k\in\N:\ F^k\in\gtq_i,\ \forall F\in\ol{\sqrt{\gta}}\right\},\\
\h(\gtq_i):=&\inf\{k\in\N:\ F^k\in\gtq_i,\ \forall F\in\sqrt{\gtq_i}\},\\ 
\h(\gta):=&\inf\left\{k\in\N:\ F^k\in\gta,\ \forall F\in\ol{\sqrt{\gta}}\right\}. 
\end{align*}
Then we have
\begin{itemize}
\item[(i)] $\h(\gta)=\sup_{i\in I}\{\h(\gtq_i,\gta)\}$;
\item[(ii)] $\sqrt{\gta}$ is closed if and only if $\h(\gta)<+\infty$;
\item[(iii)] If $\gta$ does not have immersed primary components, then $\h(\gta)=\sup_{i\in I}\{\h(\gtq_i)\}$;
\item[(iv)] $\ideal(\ceros(\gta))=\sqrt{\gta}$ if and only if $\h(\gta)<+\infty$ and if such is the case, then $\sqrt{\gta}^{\h(\gta)}\subset\gta$.
\end{itemize}
\end{thm}

To extend the Nullstellensatz to the non-closed case, we need the following characterisation of the saturation of an ideal. Namely,

\begin{deflem}\label{csaturation}
Let $(X,\an_X)$ be either a Stein space or a real-analytic space and $\gta$ an ideal of $\an(X)$. \em Define
\begin{equation*}
\begin{split}
{\mathfrak C}_1(\gta):=\{G\in\an(X): \forall \,K\subset X\text{ compact }&\ \exists\,H\in\an(X)\\
&\text{ such that }\ceros(H)\cap K=\varnothing\ \text{and}\ HG\in\gta\},\\
{\mathfrak C}_2(\gta):=\{G\in\an(X): \forall\,x\in X\,\exists\,H\in H^0(&X,\an_X)
\\
&\text{ such that }H(x)\neq0\ \text{and}\ HG\in\gta\}.
\end{split}
\end{equation*}\em
Then $\tilde{\gta}={\mathfrak C}_1(\gta)={\mathfrak C}_2(\gta)$.
\end{deflem}
\begin{proof}
As the chain of inclusions ${\mathfrak C}_1(\gta)\subset{\mathfrak C}_2(\gta)\subset\tilde{\gta}$ is clear, it only remains to check $\tilde{\gta}\subset{\mathfrak C}_1(\gta)$. 

We begin with the complex case. Let $K\subset X$ be a compact set. As $(X,\an_X)$ is a Stein space, we may assume that $K$ is holomorphically convex \cite[VII.A]{gr}. Since $\gta\an_X$ is a coherent sheaf, we deduce by Cartan's Theorem A \cite{c2} that there exists an open neighbourhood $\Omega$ of $K$ in $X$ and $A_1,\ldots,A_r\in\an(X)$ such that $\gta\an_{X,x}$ is generated as an $\an_{X,x}$-module by $A_{1,x},\ldots,A_{r,x}$ for all $x\in\Omega$. By \cite[\S2.Satz 3]{of} and Cartan's Theorem B the finitely generated ideal $\gtg:=(A_1,\ldots,A_r)\an(X)$ is saturated. By \cite[\S2.Satz 3]{of} the ideal 
$$
(\gtg:\tilde{\gta}):=\{H\in\an(X):\ H\tilde{\gta}\subset\gtg\}
$$
is saturated. Since $\gta\an_{X,x}=\tilde{\gta}\an_{X,x}$ for all $x\in X$, we deduce $(\gtg:\tilde{\gta})\an_{X,x}=\an_{X,x}$ for all $x\in\Omega$, that is, it is generated by $1$ at any point of $\Omega$. After shrinking $\Omega$, we may assume that it is an Oka--Weil neighbourhood of $K$ and that $H^0(W,(\gtg:\tilde{\gta})\an_X)=H^0(W,\an_X)$ (see \cite[VII.A.Prop.3 \& VIII.A.Prop.6]{gr}). By \cite[VIII.A.Thm.11]{gr} there exists a holomorphic function $H\in H^0(X,(\gtg:\tilde{\gta})\an_X)=(\gtg:\tilde{\gta})$ that is close to $1$ in $K$. 
Thus, $\ceros(H)\cap K=\varnothing$ and $H\tilde{\gta}\subset\gtg\subset\gta$. Therefore, we conclude $\tilde{\gta}\subset{\mathfrak C}_1(\gta)$. 

We consider the real case next. By \cite[Prop.2 \& 5]{c2} the sheaf of ideals $\gta\an_X$ can be extended to a coherent sheaf of ideals $\Fhaz$ on an open Stein neighbourhood $\Omega$ of $\R^n$ in $\C^n$. Hence the inclusion $\tilde{\gta}\subset{\mathfrak C}_1(\gta)$ follows similarly to the one of the complex case and we leave the concrete details to the reader. 
\end{proof}
\begin{remarks}\label{csaturationr}
Let $\gta\subset\gtb$ be ideals of $\an(X)$ and define ${\mathfrak R}_i(\gta):={\mathfrak C}_i(\sqrt{\gta})$ for $i=1,2$. Then 
\begin{itemize}
\item[(i)]${\mathfrak C}_i(\gta)\subset{\mathfrak C}_i(\gtb)$ and ${\mathfrak R}_i(\gta)\subset{\mathfrak R}_i(\gtb)$.
\item[(ii)] ${\mathfrak C}_i({\mathfrak C}_i(\gta))={\mathfrak C}_i(\gta)$ and ${\mathfrak R}_i({\mathfrak R}_i(\gta))={\mathfrak R}_i(\gta)$.
\end{itemize}
\end{remarks}

Now we are ready to prove Theorem \ref{perfect}.

\begin{proof}[Proof of Theorem \em \ref{perfect}]
Let us prove
$$
\ideal(\ceros(\gta))={\mathfrak R}_1(\gta)={\mathfrak R}_2(\gta)=\widetilde{\sqrt{\gta}}. 
$$ 
Clearly, ${\mathfrak R}_1(\gta)\subset{\mathfrak R}_2(\gta)\subset\widetilde{\sqrt{\gta}}\subset\ideal(\ceros(\gta))$. Thus, it remains to prove the inclusion
$$
\ideal(\ceros(\gta))\subset{\mathfrak R}_1(\gta).
$$

Assume first that $\gta$ is a closed ideal and let $K$ be a compact subset of $X$. Since $(X,\an_X)$ is a Stein space, we may assume that $K$ is holomorphically convex \cite[VII.A]{gr}. Let $\gta=\bigcap_{i\in I}\gtq_i$ be a normal primary decomposition of $\gta$. As $K$ is compact and $\{\gtq_i\}_{i\in I}$ is locally finite, the set $J:=\{i\in I:\ \ceros(\gtq_i)\cap K\neq\varnothing\}$ is finite. Let $\gta_1:=\bigcap_{i\in J}\gtq_i$ and $\gta_2:=\bigcap_{i\not\in J}\gtq_i$; clearly, $\gta=\gta_1\cap\gta_2$.

Since $K\subset X\setminus\bigcup_{i\not\in J}\ceros(\gtq_i)$ and $K$ is holomorphically convex, there exists by \cite[VII.A.Prop.3]{gr} an Oka--Weil neighbourhood $U$ of $K$ in $X\setminus\bigcup_{i\not\in J}\ceros(\gtq_i)$. By \cite[VIII.A.Thm.11]{gr} there exists a holomorphic function $H\in\gta_2=H^0(X,\gta_2\an_X)$ that is close to $1$ on $K$. On the other hand, since $\ideal(\ceros(\gtq_i))=\sqrt{\gtq}_i$ for all $i$ and there exists $m_i\geq 1$ such that $(\sqrt{\gtq}_i)^{m_i}\subset\gtp_i$ (see Theorem \ref{primary}), we find $m\geq 1$ such that $(\sqrt{\gta_1})^{m}\subset\gta_1$. Moreover, since $J$ is a finite set, we obtain
\begin{multline*}
\ideal(\ceros(\gta))=\ideal(\ceros(\gta_2\cap\gta_1))=\ideal\Big(\ceros(\gta_2\cap\bigcap_{i\in J}\gtq_i\Big)=\ideal(\ceros(\gta_2))\cap\bigcap_{i\in J}\ideal(\ceros(\gtq_i))\\
=\ideal(\ceros(\gta_2))\cap\bigcap_{i\in J}\sqrt{\gtq_i}=\ideal(\ceros(\gta_2))\cap\sqrt{\gta_1}.
\end{multline*}
If $G\in\ideal(\ceros(\gta))$, then $(HG)^{m}\in\gta_2\gta_1\subset\gta_2\cap\gta_1=\gta$, that is, $HG\in\sqrt{\gta}$ and so $\ideal(\ceros(\gta))\subset{\mathfrak R}_1(\gta)$.

For the general case, we proceed as follows. By Lemma \ref{csaturation} it holds
$$
\tilde{\gta}={\mathfrak C}_1(\gta)\subset{\mathfrak C}_1(\sqrt{\gta})={\mathfrak R}_1(\gta)
$$
and by Remarks \ref{csaturationr} we get
$$
\ideal(\ceros(\gta))=\ideal(\ceros(\tilde{\gta}))={\mathfrak R}_1(\tilde{\gta})\subset{\mathfrak R}_1({\mathfrak R}_1(\gta)) ={\mathfrak R}_1(\gta)=\widetilde {\sqrt{\gta}},
$$ 
as required.
\end{proof}

\begin{remark}\label{rperfect0}
If $\gtq$ is a primary ideal of $\an(X)$, then by Lemma \ref{lemmaprimary}
$$
\widetilde{\sqrt{\gtq}}=\begin{cases}\sqrt{\gtq}&\text{if $\gtq$ is saturated,}\\
H^0(X,\sqrt{\gtq}\an_X)&\text{otherwise.}
\end{cases}
$$
\end{remark}

\begin{examples}\label{rperfect}
(i) There are saturated ideals $\gta$ of $\an(X)$ whose radical $\sqrt{\gta}$ is not saturated. Consider the Stein space $(\C,\an_\C)$ and for each $k\geq1$ let $F,G\in\an(\C)$ be holomorphic functions whose respective zero-sets are $\N$ and such that ${\rm mult}_n(F)=n$ and ${\rm mult}_n(G)=1$ for all $n\in\N$. Observe that the ideal $\gta$ of $\an(\C)$ generated by $F$ is saturated because it is principal. However, its radical $\sqrt{\gta}$ is not saturated because $G\in\widetilde{\sqrt{\gta}}\setminus\sqrt{\gta}$.

(ii) Conversely, there are non-saturated ideals of $\an(X)$ whose radical $\sqrt{\gta}$ is saturated. Consider the Stein space $(\C,\an_\C)$ and for each $k\geq1$ let $F_k\in\an(\C)$ be a holomorphic function whose zero-set is $\N$ and such that 
$$
{\rm mult}_{n}(F_k):=\begin{cases}
1&\text{ if $n<k$,}\\
2&\text{ if $n\geq k$.}
\end{cases}
$$
Let $\gta$ be the ideal of $\an(\C)$ generated by the functions $F_k$. Let also $G\in\an(\C)$ be a holomorphic function whose zero-set is $\N$ and such that ${\rm mult}_{n}(G)=1$ for all $n\in\N$. Notice that $G^2=F_1\in\gta$ and $\sqrt{\gta}=G\an(\C)=\tilde{\gta} \neq\gta$.
\end{examples}

\section{The real Nullstellensatz in terms of \L ojasiewicz's radical}\label{s4}

We present some results relating \L ojasiewicz's radical to the real radical in the abstract setting (see also \cite{fg}).

\subsection{The real radical in the abstract setting}\label{as}

We begin by recalling some properties concerning classical Cauchy-Schwarz's inequality and Lagrange's equality. Cauchy-Schwarz's inequality says that in an Euclidean space $(E,\qq{\cdot,\cdot})$ it holds $|\qq{x,y}|\leq\|x\|\|y\|$ or equivalently $\qq{x,y}^2\leq\|x\|^2\|y\|^2$ for every couple of vectors $x,y\in E$. For $\R^n$ with its usual inner product we have
$$
(x_1y_1+\cdots+x_ny_n)^2\leq(x_1^2+\cdots+x_n^2)(y_1^2+\ldots+y_n^2)\quad\forall\, (x_1,\ldots,x_n),(y_1,\ldots,y_n)\in\R^n. 
$$
For instance, we can prove the previous inequality using the following polynomial identity in $\Z[\x,\y]:=\Z[\x_1,\ldots,\x_n;\y_1,\ldots,\y_n]$
\begin{multline*}
\Big(\sum_{i=1}^n\x_i^2\Big)\Big(\sum_{j=1}^n\y_j^2\Big)-\Big(\sum_{k=1}^n\x_k\y_k\Big)^2=\sum_{i,j=1}^n\x_i^2\y_j^2-\sum_{i,j=1}^n\x_i\y_i\x_j\y_j\\
=\sum_{\substack{i,j=1\\ i\neq j}}^n\x_i^2\y_j^2-2\sum_{\substack{i,j=1\\ i<j}}^n\x_i\y_i\x_j\y_j=\sum_{\substack{i,j=1\\ i<j}}^n(\x_i\y_j-\x_j\y_i)^2,\qquad{\rm (LE)}
\end{multline*}
which is known as Lagrange's equality. Thus, if $A$ is a (unitary commutative) ring and $a_1,\ldots,a_n$, $b_1,\ldots,b_n\in A$, it holds that
$$
\Big(\sum_{i=1}^na_i^2\Big)\Big(\sum_{j=1}^nb_j^2\Big)-\Big(\sum_{k=1}^na_kb_k\Big)^2=\sum_{\substack{i,j=1\\ i<j}}^n(a_ib_j-a_jb_i)^2\qquad{\rm (CS)}
$$
is a finite sum of squares. On the other hand, we say that an element $a\in A$ of a real ring $A$ is non-negative and we write $a\geq0$ if it belongs to all prime cones of $A$. We prove the following result, which presents the real radical in relation with \L ojasiewicz's inequality.
\begin{lem}\label{newpres}
Let $A$ be a real ring and $\gta$ an ideal of $A$. Then 
\begin{equation}\label{ri}
\sqrt[\mathsf{r}]{\gta}=\{a\in A:\ \exists\, b\in\gta,\ m\geq1 \text{ such that }\ b-a^{2m}\geq0\}.
\end{equation}
Moreover, if $\gta:=(f_1,\ldots,f_r)A$ and $f:=f_1^2+\cdots+f_r^2$, then
\begin{equation}\label{ri2}
\sqrt[\mathsf{r}]{\gta}=\{a\in A:\ \exists\, m\geq1,\ \sigma\in\Sigma A^2\ \text{ such that } \sigma f-a^{2m}\geq0\}.
\end{equation}
\end{lem}
\begin{proof}
Denote the set on the right hand side of equality \eqref{ri} with $\gtb$ and let us check $\sqrt[\mathsf{r}]{\gta}=\gtb$. Take $a\in\sqrt[\mathsf{r}]{\gta}$. There exist $a_1,\ldots,a_r\in A$ and $m\geq1$ such that 
$$
a^{2m}\leq a^{2m}+\sum_{i=1}^ra_i^2=:b\in\gta; 
$$
hence, $a\in\gtb$.

Conversely, take now $a\in\gtb$ and let $b\in\gta$ and $m\geq1$ be such that $b-a^{2m}\geq0$. Observe that there does not exist a prime cone $\alpha$ in $A$ such that $-b+a^{2m}\in\gta$ and $b-a^{2m}\not\in\supp(\alpha)$. Thus, by the abstract Positivstellensatz \cite[4.4.1]{bcr} there exist sums of squares $\sigma_1,\sigma_2$ in $A$ and a positive integer $\ell\geq 1$ such that $\sigma_1+(-b+a^{2m})\sigma_2+(-b+a^{2m})^{2\ell}=0$. Therefore
$$
(-b+a^{2m})^{2\ell}+\sigma_1+a^{2m}\sigma_2=b\sigma_2\in\gta;
$$
hence, $-b+a^{2m}\in\sqrt[\mathsf{r}]{\gta}$. As $b\in\gta\subset\sqrt[\mathsf{r}]{\gta}$ and the latter is a radical ideal, we conclude $a\in\sqrt[\mathsf{r}]{\gta}$, as required.

If $\gta=(f_1,\ldots,f_r)A$, it is clear that the set on the right hand side of equality \eqref{ri2} is contained in $\sqrt[\mathsf{r}]{\gta}$. Conversely, let $a\in\sqrt[\mathsf{r}]{\gta}$. There exist $b\in\gta$ and $\ell\geq1$ such that $b-a^{2\ell}\geq0$. Since $b\in\gta$, there exist $g_1,\ldots,g_r\in A$ such that $b=g_1f_1+\cdots g_rf_r$. By \ref{as}(CS) we get $b^2\leq f\sigma$ where $\sigma=g_1^2+\cdots+g_r^2\in\Sigma A^2$. On the other hand, since $b-a^{2\ell}\geq0$, we have
$$
b+a^{2\ell}=(b-a^{2\ell})+2a^{2\ell}\geq0,\quad\text{so}\quad b^2-a^{4\ell}=(b+a^{2\ell})(b-a^{2\ell})\geq0;
$$
hence, if we write $m:=2\ell$, we obtain $f\sigma-a^{2m}=(f\sigma-b^2)+(b^2-a^{2m})\geq0$, as required.
\end{proof}

\subsection{\L ojasiewicz's inequality and the real radical}

Recall that in the polynomial case and in the local analytic setting Artin-Lang's Theorem relates the abstract positivity of an element in the corresponding ring with its geometric positivity. More precisely,

\subsubsection{Polynomial case} 
Let $R$ be a real closed field and $X\subset R^n$ an algebraic set. Denote the \em ring of polynomial functions on $X$ \em with $R[X]:=R[\x]/\ideal(X)$ where $R[\x]:=R[\x_1,\ldots,\x_n]$ and $\ideal(X)=\{g\in R[\x]:\ g(x)=0\ \forall\,x\in X\}$. An element $f\in R[X]$ is $\geq0$ if and only if $f(x)\geq0$ for all $x\in X$. 

\subsubsection{Local analytic case} 
Let $\an_n:=\R\{\x\}:=\R\{\x_1,\ldots,\x_n\}$ and $X_a\subset\R^n_a$ be an analytic germ at a point $a\in\R^n$. Denote the \em ring of analytic function germs on $X_a$ \em with $\an(X_a):=\R\{\x-a\}/\ideal(X_a)$ where $\ideal(X_a):=\{g_a\in\R\{\x-a\}: X_a\subset\ceros(g_a)\}$. An element $f_a\in\an(X_a)$ is $\geq0$ if and only if there exist representatives $X$ of $X_a$ and $f$ of $f_a$ defined on $X$ such that $f(x)\geq0$ for all $x\in X$.

\vspace{2mm} 
We recall the well-known real Nullstellens\"atze in terms of the real radical. 

\begin{thm}[Real Nullstellens\"atze]
Let $A$ denote either $R[X]$ for an algebraic set $X$ or $\an(X_a)$ for an analytic germ $X_a\subset\R^n_a$ and let $\gta$ be an ideal of $A$. Then $\ideal(\ceros(\gta))=\sqrt[\mathsf{r}]{\gta}$.
\end{thm}
 
Now we use \L ojasiewicz's inequality in order to prove that in the ring of polynomials and the ring of germs the real radical coincides with \L ojasiewicz's radical. Since in the algebraic and the local analytic cases the geometric objects can be represented as the zero-set of a single positive semidefinite equation, it is enough to consider the cases $X:=R^n$ and $X_a:=\R^n_0$.

\begin{lem}[\L ojasiewicz's inequality]
Let $A$ denote either $R[\x]$ or $\an_n$ and let $f,g\in A$ be such that $\ceros(f)\subset\ceros(g)$. Then there exist integers $m,\ell\geq0$ and a constant $C>0$ such that $g^{2m}\leq C(1+\|\x\|^2)^\ell|f|$. In particular, if $A=\an_n$, we may take $\ell=0$.
\end{lem}

For the proof of \L ojasiewicz's inequality in the polynomial case use \cite[2.6.2 \& 2.6.6]{bcr}. For the local analytic case we refer the reader to \cite[6.4]{bm}. As a straightforward consequence of \L ojasiewicz's inequality we obtain the following descriptions of the real radical in the geometric settings we are considering. Namely,

\begin{cor}\label{pol}
Let $A$ denote either $R[X]$ for an algebraic set $X$ or $\an(X_a)$ for an analytic germ $X_a\subset\R^n_a$. Let $\gta$ be an ideal of $A$ and $f\in A$ a positive semidefinite element such that $\ceros(f)=\ceros(\gta)$. Then
$$
\ideal(\ceros(\gta))=\{g\in A:\ \exists\, m,\ell\geq0, C>0\ \text{ such that }\ C(1+\|\x\|^2)^\ell f-g^{2m}\geq0\}.
$$
In particular, if $A=\an(X_a)$, we may take $\ell=0$.
\end{cor}

\section{Real Nullstellensatz in the real-analytic setting}\label{s5} 

Let $X\subset \R^n$ be a $C$-analytic set endowed with its sheaf $\an_X$ and let $\gta\subset\an(X)$ be an ideal. If $\gta$ is finitely generated by $f_1,\ldots,f_r\in\gta$, we have seen in Lemma \ref{newpres} how to manage the function $f:=\sum_{i=1}^rf_i^2$ in the definition of \L ojasiewicz's radical, see equation \eqref{eq1}. The following result provides an analogous tool for the case when $\gta$ is not finitely generated.

\begin{lem}[Crespina Lemma]\label{crespina}
Let $\gta$ be an ideal of $\an(\R^n)$. Then there exists $f\in\tilde{\gta}$ such that 
\begin{itemize}
\item[(i)] $f$ is an infinite sum of squares in $\an(\R^n)$ and $\ceros(f)=\ceros(\gta)$.
\item[(ii)] For each $g\in\tilde{\gta}$ there exists a unit $u\in\an(\R^n)$ such that $g^2\leq fu$.
\end{itemize}
\end{lem}
\begin{proof}
By \cite[Prop.2 \& 5]{c2} the sheaf of ideals $\gta\an_{\R^n}$ can be extended to a coherent invariant sheaf of ideals $\Jhaz$ on an invariant open Stein neighbourhood $\Omega$ of $\R^n$ in $\C^n$. Let $\{L_\ell\}_{\ell\geq1}$ be an exhaustion of $\Omega$ by compact sets. Since $\Jhaz$ is invariant and coherent, we deduce by Cartan's Theorem A that there exists a countable collection of invariant holomorphic sections $\{G_j\}_{j\geq1}\subset H^0(\Omega,\Jhaz)$ such that for each $\ell\geq1$ there exists $j(\ell)$ such that for each $z\in L_\ell$ the germs $G_{1,z},\ldots,G_{j(\ell),z}$ generate the ideal $\Jhaz_z$. 

For $k\geq1$ define $\mu_k:=\max_{L_k}\{|G_k|^2\}+1$ and $\gamma_k:=1/\sqrt{2^k\mu_k}$. Consider the series $F:=\sum_{k\geq1}\gamma_k^2G_k^{2}$, which converges uniformly on the compact subsets of $\Omega$. Indeed, let $L\subset\Omega$ be a compact set and observe that there exists an index $k_0\geq1$ such that $L\subset L_k$ for all $k\geq k_0$. Moreover, for each $z\in L$ we have $\gamma_k^2|G_k(z)|^{2}\leq\frac{1}{2^k}$ if $k\geq k_0$, so
$$
\Big|\sum_{k\geq k_0}\gamma_k^2G_k^2(z)\Big|\leq\sum_{k\geq k_0}\gamma_k^2|G_k(z)|^2\leq\sum_{k\geq k_0}\frac{1}{2^k}\leq 1
$$ 
for each $z\in L$. Denote $S_m:=\sum_{k=1}^m\gamma_k^2G_k^{2}\in H^0(\Omega,\Jhaz)$; hence, $F:=\sum_{k\geq1}\gamma_k^2G_k^{2}=\lim_{m\to\infty}S_m$ in the Fr\'{e}chet topology of $H^0(\Omega,\Jhaz)$. As $H^0(\Omega,\Jhaz)$ is a closed ideal of $H^0(\Omega,\an_\Omega)$ by \cite[VIII.Thm.4, pag.60]{c1}, we conclude $F\in H^0(\Omega,\Jhaz)$, so $f:=F|_{\R^n}\in\tilde{\gta}$. For each $k\geq1$ denote $f_k:=(\gamma_kG_k)|_X$ and write $f=\sum_{k\geq1}f_k^2$. It holds $\ceros(f)=\ceros(\gta)$. Indeed,
$$
\ceros(f)=\bigcap_{k\geq1}\ceros(f_k)=\bigcap_{k\geq1}(\ceros(G_k)\cap\R^n)=\Big(\bigcap_{k\geq1}\ceros(G_k)\Big)\cap\R^n=\supp(\Jhaz)\cap\R^n=\ceros(\gta).
$$

Now let $g\in\tilde{\gta}$ and $x\in\R^n$. By the choice of the $G_k$'s and since $g_x\in\gta\an_{\R^n,x}$, there exist $a_{1,x},\ldots,a_{r,x}\in\an_{\R^n,x}$ ($r$ depends on $x$) such that $g_x=a_{1,x}f_{1,x}+\cdots+a_{r,x}f_{r,x}$. Thus, by \ref{as}(CS)
$$
g_x^2\leq\Big(\sum_{i=1}^rf_i^2\Big)_x\Big(\sum_{i=1}^ra_{i,x}^2\Big)\leq f_xM_x
$$
where $M_x$ is a positive real number such that $\sum_{i=1}^ra_{i,x}^2\leq M_x$. 

Next, pick a compact set $K\subset\R^n$ and choose a constant $M_K>0$ such that $g^2|_K\leq f|_KM_K$. Fix an exhaustion $\{K_m\}_{m\geq1}$ of $X$ by compact sets and let $u\in\an(\R^n)$ be a strictly positive analytic function such that $M_{K_m}\leq u|_{K_m\setminus K_{m-1}}$ for all $m\geq1$. Then $g^2\leq fu$, as required.
\end{proof}
\begin{remark}
Observe that in general $f\in\tilde{\gta} \setminus\gta$. Indeed, let $\gta\subset\an(X)$ be a proper ideal such that $\ceros(\gta)=\varnothing$ (see for instance Example \ref{one} in the Introduction). Then there does not exist any $f\in\gta$ such that $\ceros(f)=\ceros(\gta)$ because otherwise $\gta=\an(X)$.
\end{remark}

\begin{prop}\label{Loj} 
Let $X$ be a $C$-analytic set in $\R^n$ and $f,g\in\an(X)$ such that $\ceros(f)\subset\ceros(g)$. Let $K\subset X$ be a compact set. Then there exist an integer $m\geq1$
and an analytic function $h\in\an(X)$ such that $|h|<1$, $\ceros(h)\cap K=\varnothing$ and $|f|\geq(hg)^{2m}$.
\end{prop}
\begin{proof}
The proof of this result is contained in \cite{abs}, so we sketch the proof referring to the concrete statements in \cite{abs}. By \cite[Cor. 2.3]{abs} there exist a proper $C$-analytic subset $Y_1\subset Y:=\ceros(f)$ such that $K\cap Y_1=\emptyset$, an integer $m$ and an open neighbourhood $U$ of $Y\setminus Y_1$ contained in $X\setminus Y_1$ such that 
\begin{equation}\label{pezzo1}
g^{2m}<|f|\quad\text{on}\ U\setminus Y.
\end{equation}
We may assume $U:=\{f^2-g^{4m}>0\}$. Consider the global semianalytic set $S:=\{f^2-g^{4m}<0\}$ and its closure $\ol{S}$ in $X$. As $U$ is open and $S\cap U=\varnothing$, we get $\ol{S}\cap U=\varnothing$; hence,
$$
Y\cap\ol{S}\subset Y\setminus U\subset Y_1.
$$ 
By \cite[Thm. 2.5]{abs} there exists a positive semidefinite equation $h_0$ of $Y_1$ such that 
\begin{equation}\label{pezzo2}
h_0<|f|\quad\text{on}\ \ol{S}\setminus Y_1
\end{equation} 
and $h_0<1$ in $X$. Write $T:=\{f^2-g^{4m}=0\}$; clearly,
\begin{equation}\label{pezzo3}
X\setminus Y=(S\setminus Y)\cup(U\setminus Y)\cup(T\setminus Y).
\end{equation}
As the non-negative functions $h_0$, $\frac{|g|}{1+g^2}$ and $\frac{1}{1+g^2}$ are strictly smaller than $1$ on $X$, one deduces using equations \eqref{pezzo1}, \eqref{pezzo2} and \eqref{pezzo3}
$$
\Big(h_0\frac{g}{1+g^2}\Big)^{2m}<|f|
$$
on $X\setminus Y$. Thus, taking $h:=\frac{h_0}{1+g^2}$, we are done.
\end{proof}

Now we are ready to prove Theorem \ref{perfect2}.

\begin{proof}[Proof of Theorem \em\ref{perfect2}]
Following the notations of Definition \ref{csaturation}, consider the ideals
\begin{equation*}
\begin{split}
{\mathfrak L}_1(\gta):={\mathfrak C}_1({\sqrt[\text{\L}]{\gta}}) \quad\text{and}\quad {\mathfrak L}_2(\gta):={\mathfrak C}_2({\sqrt[\text{\L}]{\gta}}).
\end{split}
\end{equation*}
By Lemma \ref{csaturation} we have ${\mathfrak L}_1(\gta)={\mathfrak L}_2(\gta)=\widetilde{\sqrt[\text{\L}]{\gta}}$. Recall that by Remark \ref{csaturationr} (ii) it holds ${\mathfrak L}_i({\mathfrak L}_i(\gta))={\mathfrak L}_i(\gta)$. We want to show $\ideal(\ceros(\gta))=\widetilde{\sqrt[\text{\L}]{\gta}}$. Clearly $\widetilde{\sqrt[\text{\L}]{\gta}}\subset\ideal(\ceros(\gta))$, so it is enough to prove the inclusion $\ideal(\ceros(\gta))\subset{\mathfrak L}_1(\gta)$. 

Assume first that $\gta$ is a saturated ideal. By Lemma \ref{crespina} there exists a positive semidefinite $f\in\gta$ such that $\ceros(f)=\ceros(\gta)$. Let $g\in\ideal(\ceros(\gta))$ and $K\subset X$ be a compact set. By Proposition \ref{Loj} there exist an integer $m\geq1$ and an analytic function $h\in\an(X)$ such that $\ceros(h)\cap K=\varnothing$ and $f\geq(hg)^{2m}$, that is, $hg\in\sqrt[\text{\L}]{\gta}$. Thus,
$g\in{\mathfrak L}_1(\gta)$, so $\ideal(\ceros(\gta))\subset{\mathfrak L}_1(\gta)$.

For the general case we proceed as follows. By Lemma \ref{csaturation} we obtain $\tilde{\gta}={\mathfrak C}_1(\gta)\subset{\mathfrak C}_1({\sqrt[\text{\L}]{\gta}})={\mathfrak L}_1(\gta)$; hence, 
$$
\ideal(\ceros(\gta))=\ideal(\ceros(\tilde{\gta}))={\mathfrak L}_1(\tilde{\gta})={\mathfrak L}_1({\mathfrak C}_1 ({\gta}))\subset{\mathfrak L}_1({\mathfrak C}_1 ({\sqrt[\text{\L}]{\gta}}))={\mathfrak L}_1({\mathfrak L}_1(\gta))={\mathfrak L}_1(\gta),
$$ 
as required.
\end{proof}
\begin{remark}
We use the notations of the previous proof. If $\gta$ is saturated and $f\in\gta$ satisfies $\ceros(\gta)=\ceros(f)$, then
$$
\ideal(\ceros(\gta))=\ideal(\ceros(f^2))={\mathfrak L}_1(f^2\an(X))={\mathfrak L}_2(f^2\an(X)).
$$
The previous equality can be understood as the counterpart of Lemma \ref{newpres} and Corollary \ref{pol} in the $C$-analytic setting.
\end{remark}

\subsection{Convex ideals} 
We introduce this concept to relate \L ojasiewicz's radical with the classical radical. An ideal $\gta$ of $\an(X)$ is \em convex \em if each $g\in\an(X)$ satisfying $|g|\leq f$ for some $f\in\gta$ belongs to $\gta$. In particular, \L ojasiewicz's radical is a radical convex ideal. Moreover, we define the \em convex hull $\gtg(\gta)$ of an ideal $\gta$ of $\an(X)$ \em by
$$
\gtg(\gta):=\{g\in\an(X):\ \exists f\in\gta\text{ such that }|g|\leq f\}.
$$
Notice that $\gtg(\gta)$ is the smallest convex ideal of $\an(X)$ that contains $\gta$ and $\sqrt[\text{\L}]{\gta}=\sqrt{\gtg(\gta)}$.

\begin{remark}
If $\gta$ is a convex ideal of the ring $\an(X)$, then its radical $\sqrt{\gta}$ is also a convex ideal of this ring. 

Indeed, let $f\in\sqrt{\gta}$ and $g\in\an(X)$ be such that $|g|\leq f$. Let $m\geq1$ be such that $f^m\in\gta$. Clearly, $|g^m|\leq f^m$, so $g^m\in\gta$ and $g\in\sqrt{\gta}$.
\end{remark}

\begin{examples}\label{rperfect2}
(i) There exist saturated ideals $\gta$ of $\an(X)$ whose \L ojasiewicz radical $\sqrt[\text{\L}]{\gta}$ is not saturated and there exist non-saturated ideals of $\an(X)$ whose \L ojasiewicz radical $\sqrt[\text{\L}]{\gta}$ is saturated. Consider the Examples \ref{rperfect} after substituting $\C$ by $\R$. 

(ii) There exist convex saturated ideals that are not radical. Take $\gta:=(x^2,xy,y^2)\an(\R^2)$.

(iii) There exist radical saturated ideals that are not convex. Let $\gta:=(x^2+y^2)\an(\R^2)$.

(iv) There exist convex radical ideals that are not saturated. Indeed, let $f_1,f_2\in\an(\R)$ be such that $\ceros(f_1)=\ceros(f_2)=\N$ and ${\rm mult}_\ell(f_1)=\ell$ and ${\rm mult}_\ell(f_2)=1$ for all $\ell\geq1$. Let $\gta:=\sqrt[\text{\L}]{f_1^2\an(\R^3)}$, which is a radical convex ideal of $\an(\R)$. However, it is not saturated. 

Otherwise we obtain $\gta=\ideal(\ceros(\gta))$ by Theorem \ref{perfect2}, so $f_2\in\gta$. Thus, there exist $m\geq 1$ and $a\in\an(\R)$ such that $h:=af_1^2-f_2^{2m}\geq0$. As $h\geq0$, for each $\ell\geq1$ it holds
\begin{multline*}
2\ell\leq{\rm mult}_{\ell}(a)+2{\rm mult}_{\ell}(f_1)={\rm mult}_{\ell}(af_1^2)
={\rm mult}_{\ell}(h+f_2^{2m})\\
=\min\{{\rm mult}_{\ell}(h),2m{\rm mult}_{\ell}(f_2)\}=\min\{{\rm mult}_{\ell}(h),2m\}\leq 2m,
\end{multline*}
which is a contradiction. Therefore $\gta$ is not saturated.
\end{examples}

\begin{cor}\label{ahora}
Let $X\subset\R^n$ be a $C$-analytic set and $\gta$ a convex saturated ideal of $\an(X)$. Let $\gta=\bigcap_{i\in I}\gtq_i$ be a normal primary decomposition of $\gta$ and $J$ the collection of the indices corresponding to the isolated primary components of $\gta$. Then
\begin{itemize}
\item[(i)] If $i_0\in J$, $\sqrt{\gtq_{i_0}}$ is a convex saturated prime ideal.
\item[(ii)] $\widetilde{\sqrt[\text{\L}]{\gta}}=\bigcap_{j\in J}\sqrt{\gtq_j}=\widetilde{\sqrt{\gta}}$.
\end{itemize}
\end{cor}
\begin{proof}
(i) By Lemma \ref{lemmaprimary}(ii) we know that $\sqrt{\gtq_{i_0}}$ is a prime saturated ideal (because so is $\gtq_{i_0}$). We claim: \em There exists $h_{i_0}\in\bigcap_{j\neq i_0}\gtq_j\setminus\sqrt{\gtq_{i_0}}$\em. Otherwise, by Lemma \ref{db1210} there exists $j\neq i_0$ such that $\gtq_i\subset{\gtq_{i_0}}$; hence, $\sqrt{\gtq_j}\subset\sqrt{\gtq_{i_0}}$ and by the minimality of $\sqrt{\gtq_{i_0}}$ we deduce $\sqrt{\gtq_j}=\sqrt{\gtq_{i_0}}$, which contradicts the fact that the primary decomposition is normal. 

Fix $h_{i_0}\in\bigcap_{j\neq i_0}\gtq_j\setminus\sqrt{\gtq_{i_0}}$ and let $g\in\an(X)$ be such that $|g|\leq f$ for some $f\in\sqrt{\gtq_{i_0}}$. Then $|gh_{i_0}^2|\leq fh_{i_0}^2$ and since $\gta$ is convex and $fh_{i_0}^2\in\bigcap_{i\in I}\gtq_i=\gta$, we deduce $gh_{i_0}^2\in\gta\subset\gtq_{i_0}$; hence, we conclude $g\in\gtq_{i_0}$ because $\gtq_{i_0}$ is primary and $h_{i_0}\not\in\sqrt{\gtq_{i_0}}$.

(ii) By Theorem \ref{perfect2} and Remark \ref{int} we conclude
$$
\widetilde{\sqrt[\text{\L}]{\gta}}=\ideal(\ceros(\gta))=\bigcap_{j\in J}\ideal(\ceros(\gtq_j))=\bigcap_{j\in J}\sqrt{\gtq_j}=\bigcap_{i\in I}\sqrt{\gtq_i}=\widetilde{\sqrt{\gta}},
$$
as required.
\end{proof}

\begin{examples}\label{ahora2}
(i) The primary ideal $\gtq:=(x^2,y^2)\an(\R^2)$ is not convex while $\sqrt{\gtq}=(x,y)\an(\R^2)=\ideal(\ceros(\gtq))$ is convex. The functions $f:=x^2+y^2\in\gtq$ and $g:=xy\in\an(\R^2)$ satisfy $|g|\leq f$ but $g\not\in\gtq$. Thus, $\gtq$ is not convex.

(ii) Under the hypotheses of Corollary \ref{ahora} the corresponding result is no longer true if $\sqrt{\gtq_{i_0}}$ is the radical of an immersed primary component of $\gta$. Let $\gta:=\gtq_1\cap\gtq_2=(z^3(x^2+y^2),z^4)\an(\R^3)$ be the intersection of the primary ideals $\gtq_1:=z^3\an(\R^3)$ and $\gtq_2:=(x^2+y^2,z^4)\an(\R^3)$. Observe that $\sqrt{\gtq_1}\subsetneq\sqrt{\gtq_2}$. Let us check that $\gta$ is convex while $\gtq_2$ is not convex. 

Indeed, if $f\in\gta$ is positive semidefinite, then $z^4$ divides $f$. If $g\in\an(\R^3)$ satifies $|g|\leq f$, then $z^4$ divides $g$, so $g\in\gta$; hence, $\gta$ is convex. However, $\gtq_2$ is not convex because $x^2\leq x^2+y^2$ but $x^2\not\in\gtq_2$.
\end{examples}

By Theorem \ref{perfect2} the equality $\ideal(\ceros(\gta))=\widetilde{\sqrt{\gtg(\gta)}}$ holds for each ideal $\gta$ of $\an(X)$ where $X\subset\R^n$ is a $C$-analytic set. The last part of this section will be dedicated to determine how the operations $\widetilde{\cdot}$, $\sqrt{\cdot}$ and $\gtg(\cdot)$ commute.

\begin{lem}\label{cs}
Let $X\subset\R^n$ be a $C$-analytic set and $\gta$ an ideal of $\an(X)$. Then
\begin{itemize}
\item[(i)] If $\gta$ is convex, $\tilde{\gta}$ is also convex.
\item[(ii)] There exists $f\in\tilde{\gta}$ such that $(\widetilde{\gtg{(\tilde{\mbox{\!\!$\gta$\,\,}})}})^2\subset\gtg(f\an(X))\subset\gtg(\tilde{\gta})$. In particular, 
$$
\sqrt{\widetilde{\gtg(\text{$\tilde{\gta}$})}}\,\,\,=\sqrt{\gtg(f\an(X))}=\sqrt{\gtg(\tilde{\gta})}=\sqrt[\text{\em \L}]{\tilde{\gta}}.
$$
\end{itemize}
\end{lem}
\begin{proof}
(i) Let $g\in\an(X)$ and $f\in\tilde{\gta}$ be such that $|g|\leq f$. Let $K\subset X$ be a compact set. By Lemma \ref{csaturation} there exists $h\in\an(X)$ such that $\ceros(h)\cap K=\varnothing$ and $h^2f\in\gta$. As $|h^2g|\leq h^2f$ and $\gta$ is convex, we deduce $h^2g\in\gta$. By Lemma \ref{csaturation} we obtain $g\in\tilde{\gta}$, so $\tilde{\gta}$ is convex.

(ii) As $\an(X)=\an(\R^n)/\ideal(X)$, there exists an ideal $\gtb$ of $\an(\R^n)$ that contains $\ideal(X)$ such that $\gta=\gtb/\ideal(X)$. For the sake of clearness we denote the elements of $\an(X)=\an(\R^n)/\ideal(X)$ with $\widehat{h}:=h+\ideal(X)$. By Lemma \ref{crespina} there exists $f\in\tilde{\gtb}$ such that for each $a\in\tilde{\gtb}$ there exists a unit $u\in\an(\R^n)$ satisfying $a^2\leq fu$; hence, $\widehat{a}^2\leq\widehat{f}\widehat{u}$.

Pick $\widehat{g}\in\widetilde{\gtg(\text{$\tilde{\gta}$})}\ $ and let $K\subset X$ be a compact set. Then there exists a function $\widehat{h}_K\in\an(X)$ such that $\ceros(\widehat{h}_K)\cap K=\varnothing$ and $\widehat{h}_K\widehat{g}\in\gtg(\text{$\tilde{\gta}$})$; hence, there exists $\widehat{a}_K\in\text{$\tilde{\gta}$}$ such that $|\widehat{h}_K\widehat{g}|\leq\widehat{a}_K$.

Let $u_K\in\an(\R^n)$ be a unit such that $a_K^2\leq fu_K$ and let $M_K>0$ be such that $\widehat{g}^2|_K\leq\widehat{f}|_KM_K$ (recall that $|\widehat{h}_K\widehat{g}|\leq\widehat{a}_K$ and $\ceros(\widehat{h}_K)\cap K=\varnothing$). Fix an exhaustion $\{K_m\}_{m\geq1}$ of $X$ by compact sets and let $\widehat{u}\in\an(X)$ be a strictly positive analytic function such that $M_{K_m}\leq\widehat{u}|_{K_m\setminus K_{m-1}}$ for all $m\geq1$. Then $\widehat{g}^2\leq \widehat{f}\widehat{u}$; hence, $\widehat{g}^2\in\gtg(\widehat{f}\an(X))$.

To finish observe the following: If $\widehat{g}_1,\widehat{g}_2\in\widetilde{\gtg(\text{$\tilde{\gta}$})}$\,\,\,, then there exist strictly positive analytic functions $\widehat{u}_1,\widehat{u}_2\in\an(X)$ such that $\widehat{g}_i^2\leq\widehat{f}\widehat{u}_i^2$ for $i=1,2$; hence, $|\widehat{g}_1\widehat{g}_2|\leq\widehat{g}_1^2+\widehat{g}_2^2\leq\widehat{f}(\widehat{u}_1^2+\widehat{u}_2^2)$, so $\widehat{g}_1\widehat{g}_2\in\gtg(\widehat{f}\an(X))$. Thus, $(\widetilde{\gtg(\text{$\tilde{\gta}$})}\,\,)^2\subset\gtg(f\an(X))$, as required.
\end{proof}
\begin{remark}\label{49}
If we are working in the framework of convex saturated ideals, it holds an analogous result to Theorem \ref{fo} when substituting `Stein space' by `$C$-analytic set' and `closed ideal' by `convex saturated ideal'. The proof runs analogously to the one of Theorem \ref{fo} (\cite[\S5.Satz 9]{of}) and we leave the concrete details to the reader.
\end{remark}

\section{The real-analytic radical and the real Nullstellensatz}\label{s6}

In this section we prove Theorem \ref{h17nss}, that is, we relate the real Nullstellensatz with the classical real radical by means of the representation of positive semidefinite functions as sums of squares of meromorphic functions. We begin by recalling the definition of $H$-sets and $H^{\mathsf a}$-sets and presenting some properties.

\begin{defn}\label{rr}
A $C$-analytic set $Z\subset\R^n$ is an {\em $H$-set} if each positive semidefinite analytic function
$f\in\an(\R^n)$ whose zero-set is $Z$ can be represented as a sum of squares of meromorphic functions on $\R^n$. More generally, we say that $Z$ is an {\em $H^{\mathsf a}$-set} if such representation may involve infinitely many squares.
\end{defn}

The following properties are stated and proved for $H^{\mathsf a}$-sets but many of them work analogously for $H$-sets.

\begin{remarks}\label{rrr}
(i) Let $Y\subset Z\subset\R^n$ be $C$-analytic sets. If $Z$ is an $H^{\mathsf a}$-set, then $Y$ is also an $H^{\mathsf a}$-set.

Indeed, let $f\in\an(\R^n)$ be a positive semidefinite analytic function such that $\ceros(f)=Y$. Let now $g\in\an(\R^n)$ be an analytic function such that $\ceros(g)=Z$. Observe that $h:=g^2f$ is positive semidefinite and $\ceros(h)=Z$; hence, $h$ is a sum of squares of meromorphic functions on $\R^n$, so the same happens for $f$. Thus, $Y$ is an $H^{\mathsf a}$-set.

(ii) If $Z\subset\R^n$ is an $H$-set, the same holds for each global irreducible component of $Z$.

(iii) Let $Z\subset\R^n$ be a $C$-analytic set. Then $Z$ is an $H^{\mathsf a}$-set if and only if there exists a positive semidefinite $f\in\an(\R^n)$ such that $\ceros(f)=Z$ and each $h\in\an(\R^n)$ with $\ceros(h)=Z$ and $0\leq h\leq f$ is a sum of squares of meromorphic functions on $\R^n$.

\begin{proof}
The `only if' implication is clear. Conversely, assume there exists a positive semidefinite analytic equation $f$ of $Z$ with the property in the statement and let $g\in\an(\R^n)$ be another positive equation of $Z$. Observe
$$
f-\Big(\frac{f}{\sqrt{1+fg}}\Big)^2g=f\Big(1-\frac{fg}{1+fg}\Big)\geq0\quad\text{and}\quad\ceros
\Big(\Big(\frac{f}{\sqrt{1+fg}}\Big)^2g\Big)=Z.
$$ 
Thus, $0\leq h:=(\frac{f}{\sqrt{1+fg}})^2g\leq f$ and $h$ is by hypothesis a sum of squares of meromorphic functions on $\R^n$, so the same happens with $g$. Thus, $Z$ is an $H^{\mathsf a}$-set.
\end{proof}

(iv) By \cite{j2,rz1} each compact $C$-analytic subset of $\R^n$ is an $H$-set. Therefore, by \cite[1.9]{abfr3} each $C$-analytic set $Z$ whose connected components are compact is an $H^{\mathsf a}$-set.

(v) Let $Z$ be a $C$-analytic set. By \cite[1.2]{abf} we obtain that $Z$ is an $H^{\mathsf a}$-set if and only if each global irreducible function $f\in\an(\R^n)$ with $\ceros(f)\subset Z$ is a sum of squares of meromorphic functions on $\R^n$.

(vi) Hilbert's 17th Problem in its more general formulation involving infinite sums of squares has a positive answer for $\an(\R^n)$ if and only if all connected $C$-analytic subsets of $\R^n$ of dimensions $1\leq d\leq n-2$ are $H^{\mathsf a}$-sets. Recall that given a $C$-analytic set $Z\subset\R^n$ of codimension $\geq2$, there exists by \cite{dem} an irreducible analytic function $f\in\an(\R^n)$ whose zero-set is $Z$. 
\end{remarks}

In \cite[Lem. 4.1]{abfr3} we developed a procedure to move the remainder $\ceros(b)\setminus\ceros(f)$ of the zero set of the denominator $b$ in a representation of a positive semidefinite analytic function $f$ as a sum of squares of meromorphic functions while $f$ was kept invariant (up to multiplication by a unit $u\in\an(\R^n)$). This tool was crucial to eliminate the remainder $\ceros(b)\setminus\ceros(f)$. The following result, in analogy to \cite[Lem. 4.1]{abfr3}, is used in the proof of Theorem \ref{h17nss} to perturb the complex part of the zero set $\ceros(B)$ of a holomorphic extension $B$ of $b$ while $f$ is again kept invariant (up to multiplication by a unit $u\in\an(\R^n)$). This is the clue to prove in Theorem \ref{h17nss} that if $\gtp$ is a real saturated prime ideal whose zero set is an $H$-set, then $\ideal(\ceros(\gtp))=\gtp$.

\begin{lem}[Perturbing denominators]\label{move}
Let $b,f\in\an(\R^n)$ be non-constant analytic functions. Let $\Omega$ be an invariant open neighbourhood $\Omega$ of $\R^n$ in $\C^n$ to which both $b,f$ extend to holomorphic functions $B,F$. Let $Z\subset\Omega$ be a complex analytic set such that $Z_{x_0}\not\subset\ceros(F)_{x_0}$ for some $x_0\in\ceros(f)$. Then there exists an analytic diffeomorphism $\varphi:\R^n\to\R^n$ such that:
\begin{itemize}
\item[(i)] $f\circ\varphi=fu$ for some unit $u\in\an(\R^n)$.
\item[(ii)] $Z_{x_0}\not\subset\ceros(B_0)_{x_0}$ where $B_0:\Omega_0\to\C$ is the holomorphic extension of $b_0:=b\circ\varphi$ to a small enough open neighbourhood $\Omega_0\subset\Omega$ of $\R^n$ in $\C^n$.
\end{itemize}
\end{lem}
\begin{proof}
We may assume that $b$ can be extended to a holomorphic function $B$ on $\Omega$ and $Z_{x_0}\subset\ceros(B)_{x_0}$ because otherwise we choose $\varphi=\id$ and are done. The proof is conducted in two steps:\setcounter{substep}{0}

\vspace{2mm}
\noindent{\sc Step 1}. \em We construct a family of analytic diffeomorphisms $\phi_\lambda:\R^n\to\R^n$ depending on a parameter $\lambda\in[-1,1]^n$ that satisfy condition \em (i) \em in the statement\em.

Fix a strictly positive analytic function $\veps\in\an(\R^n)$ and for each $\lambda:=(\lambda_1,\ldots,\lambda_n)\in[-1,1]^n$ consider the analytic map 
$$
\phi_\lambda:\R^n\to\R^n,\ x\mapsto x+f^2(x)\veps(x)\lambda.
$$ 
We choose $\veps$ small enough in such a way that $\phi_\lambda$ is by \cite[2.1.7]{h} an analytic diffeomorphism for each $\lambda\in[-1,1]^n$. Since the function 
$$
f_0:\R^n\times\R^n\times\R\to\R,\ (x,y,t)\mapsto f(x+ty)-f(x)
$$
vanishes identically on the set $\R^n\times\R^n\times\{0\}$, there exists an analytic $h\in\an(\R^n\times\R^n\times\R)$ such that $f_0=ht$. Thus,
\begin{equation}\label{ulam}
f\circ\phi_\lambda(x)=f(x)+f(x)^2\veps(x)h(x,\lambda,f(x)^2\veps(x))=f(x)u_\lambda(x)
\end{equation}
where $u_\lambda(x):=1+f(x)\veps(x)h(x,\lambda,f(x)^2\veps(x))$. Note that $\ceros(f\circ\phi_\lambda)\subset\ceros(f)$, so $\ceros(f\circ\phi_\lambda)=\ceros(f)$. 

Indeed, if $x\in\R^n$ satisfies $f\circ\phi_\lambda(x)=0$, then $y:=\phi_\lambda(x)\in\ceros(f)$. Since $\phi_\lambda$ is bijective and $\phi_\lambda(y)=y$ (because $f(y)=0$), we deduce $x=y\in\ceros(f)$.

By its definition $u_\lambda$ is a unit in a neighbourhood of $\ceros(f)$ and does not vanish outside $\ceros(f)=\ceros(f\circ\phi_\lambda)$ (see equation \eqref{ulam}), so we conclude that $u_\lambda$ is a unit in $\an(\R^n)$ for all $\lambda\in[-1,1]^n$. Therefore the diffeomorphisms $\phi_\lambda$ satisfy condition (i) for all $\lambda\in[-1,1]^n$. 

\vspace{2mm}
\noindent{\sc Step 2}. \em We find now $\lambda_0\in[-1,1]^n$ such that $\varphi:=\phi_{\lambda_0}$ also satisfies condition \em (ii). Consider the family of diffeomorphisms $\phi_\lambda$ as the analytic map
$$
\phi:\R^n\times[-1,1]^n\to\R^n,\ (x,\lambda)\mapsto\phi_\lambda(x).
$$
After shrinking $\Omega$, we may assume that $\veps,b$ can be extended holomorphically to $E,B\in\an(\Omega)$ and $\Omega$ is connected. Thus, $\phi$ can be extended to the holomorphic map 
$$
\Phi:\Omega\times\C^n\to\C^n,\ (z,\mu)\mapsto z+F^2(z)E(z)\mu.
$$ 
Let $U:=\Phi^{-1}(\Omega)$ and consider the holomorphic function
$$
B\circ\Phi:U\to\C,\ (w,\mu)\mapsto B\circ\Phi(w,\mu)=B\circ\Phi_\mu(w).
$$
Fix a polydisc $\Delta_0\times\Delta_1\subset\Omega\times\C^n$ of center $(x_0,0)$ and radius $0<\rho<1$ contained in $U$. Then it holds:

\begin{substeps}{move}\label{move0}
\em The map $(B\circ\Phi)_w:\Delta_1\to\C,\ \mu\mapsto(B\circ\Phi)(w,\mu)$ is not identically zero for each $w\in\Delta_0$\em.
\end{substeps}

Otherwise there exists $w\in\Delta_0$ such that
$$
(B\circ\Phi)_w(\mu):=B\circ\Phi(w,\mu)=B(w+F^2(w)E(w)\mu)
$$
is identically zero on the polydisc $\Delta_1$. By the Identity Principle we deduce that $B$ is identically zero, which contradicts the hypothesis that $b$ is not constant.

Since $Z_{x_0}\not\subset\ceros(F)_{x_0}$, by the complex curve selection lemma there exists a complex analytic curve $\gamma:{\mathbb D}_\delta\to Z$ (defined on the disc ${\mathbb D}_\delta$) such that $\gamma({\mathbb D}_\delta)\subset\Delta_0$, $\gamma(0)=x_0$ and $\gamma(s)\not\in\ceros(F)$ for all $s\neq0$. Consider the holomorphic function 
$$
G:{\mathbb D}_\delta\times\Delta_1\to\C,\ (s,\mu)\mapsto(B\circ\Phi)(\gamma(s),\mu).
$$
We know by \ref{move}.\ref{move0} that the holomorphic function $G_s:\Delta_1\to\C,\ \mu\mapsto G(s,\mu)$ is not identically zero for each $s\in{\mathbb D}_\delta$. Choose now a sequence $\{s_k\}_k\subset{\mathbb D}_\delta$ converging to $0$ and observe that for each $k$ the set $W_k:=(\Delta_1\cap\R^n)\setminus\ceros(G_{s_k})=[-\rho,\rho]^n\setminus\ceros(G_{s_k})$ is open and dense in $\Delta_1\cap\R^n=[-\rho,\rho]^n$ because each $\ceros(G_{s_k})$ is a proper analytic subset of $\Delta_1$. By Baire's Theorem the intersection $W:=\bigcap_{k\geq1}W_k$ is dense in $\Delta_1\cap\R^n$ and we choose $\lambda_0\in W$.

If $b_0:=b\circ\phi_{\lambda_0}$, then $B_0:=B\circ\Phi_{\lambda_0}$ is its holomorphic extension to $\Omega$ where $\Phi_{\lambda_0}:\Omega\to\C^n,\ z\mapsto\Phi(z,\lambda_0)$. By the choice of $\lambda_0$ we have $B_0\circ\gamma(s_k)\neq0$ for all $k\geq1$; hence $B_0\circ\gamma$ is not identically zero on ${\mathbb D}_{\delta}$, so the germ $(B_0\circ\gamma)_0\neq0$. We conclude $Z_{x_0}\not\subset\ceros(B_0)_{x_0}$, as required.
\end{proof}

Once this is proved, we approach the proof of Theorem \ref{h17nss} when $\ceros(\gta)$ is an $H^{\mathsf a}$-set. The proof is similar if $\ceros(\gta)$ is an $H$-set.

\begin{proof}[Proof of Theorem {\em\ref{h17nss}}]
The proof is conducted in several steps:\setcounter{substep}{0}

\vspace{2mm}
\noindent{\sc Step 1}. Assume first that $\gta=\gtp$ is in addition a prime ideal. Since $\an(X)=\an(\R^n)/\ideal(X)$, we may assume by the correspondence theorem for ideals that $\gtp$ is a saturated real prime ideal of $\an(\R^n)$. Observe that the `only if' implication is clear since $\ideal(\ceros(\gtp))$ is real-analytic and saturated. For the converse, we proceed as follows. By \cite[Prop.2 \& 5]{c2} the sheaf of ideals $\gtp\an_{\R^n}$ can be extended to a coherent sheaf of ideals $\Jhaz$ on an invariant connected open Stein neighbourhood $\Omega$ of $\R^n$ in $\C^n$. Recall that $\Jhaz_x=\gtp\an_{\R^n,x}\otimes\C=\gtp\an_{\C^n,x}$ for all $x\in\R^n$ and that as $\gtp$ is saturated, $\gtp=H^0(\R^n,\gtp\an_{\R^n})$. Denote the support in $\Omega$ of $\Jhaz$ with $Z:=\{z\in\Omega:\ \Jhaz_z\neq\an_{\C^n,z}\}$. Let us check: \em $\gtb:=H^0(\Omega,\Jhaz)$ is a prime closed ideal of $\an(\Omega)$ such that $Z(\gtb)=Z$\em.

To prove the primality of $\gtb$, pick $F_1,F_2\in\an(\Omega)$ such that $F_1F_2\in\gtb$. We write $F_i:=\Re(F_i)+\sqrt{-1}\Im(F_i)$ and observe 
$$
(\Re(F_1)^2+\Im(F_1)^2)(\Re(F_2)^2+\Im(F_2)^2)=F_1F_2(\ol{F_1\circ\sigma})(\ol{F_2\circ\sigma})\in\gtb.
$$
We deduce
$$
(\Re(F_1)^2+\Im(F_1)^2)|_{\R^n}(\Re(F_2)^2+\Im(F_2)^2)|_{\R^n}\in H^0(\R^n,\gtp\an_{\R^n})=\gtp.
$$
As $\gtp$ is a real prime ideal, we may assume $\Re(F_1)|_{\R^n},\Im(F_1)|_{\R^n}\in\gtp$, so $\Re(F_1),\Im(F_1)\in\gtb$; hence, $F_1=\Re(F_1)+\sqrt{-1}\Im(F_1)\in\gtb$. Thus, $\gtb$ is prime. Of course, as $\varnothing\neq Z(\gtp)\subset Z(\gtb)$, we deduce by Lemma \ref{lemmaprimary} that $\gtb$ is closed. The equality $Z(\gtb)=Z$ holds because $Z$ is the support of the coherent sheaf of ideals $\Jhaz$.

Suppose now by contradiction that there exists a function $g\in\ideal(\ceros(\gtp))\setminus\gtp$. After shrinking $\Omega$ if necessary, we may assume that $g$ can be extended to a holomorphic function $G$ on $\Omega$. We claim: \em There exists a point $x_0\in\R^n$ such that $Z_{x_0}\not\subset\ceros(G)_{x_0}$ while $\ceros(\gtp)\subset\ceros(g)$\em.

Indeed, if $Z_x\subset\ceros(G)_x$ for each $x\in\R^n$, we may assume $Z\subset\ceros(G)$ after shrinking $\Omega$ if necessary; hence, $G\in\ideal(Z)=\ideal(\ceros(\gtb))=\gtb$ because $\gtb$ is a closed prime ideal of $\an(\Omega)$. Thus, $g\in\gtp$, which is a contradiction. Consequently there exists $x_0\in\R^n$ such that $Z_{x_0}\not\subset\ceros(G)_{x_0}$.

By Proposition \ref{Loj} there exist $f\in\gtp$, $h\in\an(\R^n)$ and $m\geq1$ such that $h(x_0)\neq0$ and $f_0:=f-h^2g^{2m}\geq0$. As $h(x_0)\neq0$, we have $h\not\in\gtp$. Substitute $f_0$ by $f_1:=f-h_1^2g^{2m}$ where $h_1:=\frac{h}{\sqrt{1+h^2g^{2m}}}$ in order to have $\ceros(f_1)=\ceros(f)$, which is an $H^{\mathsf a}$-set. 

Indeed, as $h_1\leq h$, it holds $f_1\geq0$. Since $f-h^2g^{2m}\geq0$ and therefore $f\geq0$, we have
$$
\ceros(f_1)=\ceros((f-h^2g^{2m})+(fh^2g^{2m}))=\ceros(f-h^2g^{2m})\cap\ceros(fh^2g^{2m})=\ceros(f)\cap\ceros(hg)=\ceros(f).
$$
Since $\ceros(\gtp)=\ceros(f_1)$ is an $H^{\mathsf a}$-set, there exists a not identically zero $b\in\an(\R^n)$ such that $b^2f_1=\sum_{i\geq1}a_i^2$ for some $a_i\in\an(\R^n)$. 

After shrinking $\Omega$, $f_1,h_1$ can be extended to holomorphic functions $F_1,H_1:\Omega\to\C$. In order to apply Lemma \ref{move} to $b,f_1,Z$ and $\Omega$ we show first that $Z_{x_0}\not\subset\ceros(F_1)_{x_0}$. Otherwise, as $F\in\gtb$ and $H_1(x_0)\neq0$,
$$
Z_{x_0}\subset\ceros(F)_{x_0}\cap\ceros(F_1)_{x_0}\subset\ceros(F-F_1)_{x_0}=\ceros(H_1^2G^{2m})_{x_0}=\ceros(H_1)_{x_0}\cup\ceros(G)_{x_0}=\ceros(G)_{x_0},
$$
which is a contradiction.

By Lemma \ref{move} there exists an analytic diffeomorphism $\varphi:\R^n\to\R^n$ such that:
\begin{itemize}
\item[(i)] $f_1\circ\varphi=f_1u$ for some unit $u\in\an(\R^n)$.
\item[(ii)] $Z_{x_0}\not\subset\ceros(B_1)_{x_0}$ where $B_1:\Omega_0\to\C$ is the holomorphic extension of $b_1:=b\circ\varphi$ to a small enough open neighbourhood $\Omega_0\subset\Omega$ of $\R^n$ in $\C^n$.
\end{itemize}

Let $v\in\an(\R^n)$ be a strictly positive unit such that $v^2=u^{-1}$; then
$$
b_1^2f=b_1^2h_1^2g^{2m}+b_1^2f_1=b_1^2h_1^2g^{2m}+\sum_{i\geq1}((a_i\circ\varphi)v)^2.
$$
Observe that since $Z_{x_0}\not\subset\ceros(B_1)_{x_0}$, we have $b_1\not\in\gtp$. As $f\in\gtp$ and $\gtp$ is a real-analytic ideal, we deduce $b_1h_1g^m\in\gtp$, which contradicts the fact that $b_1,h_1,g\not\in\gtp$. We conclude $\ideal(\ceros(\gtp))=\gtp$, as required.

\vspace{2mm}
\noindent{\sc Step 2}. Now assume that $\gta$ is a saturated real-analytic ideal of $\an(X)$ whose zero-set is an $H$-set. By Proposition \ref{propproprieta} and Corollary \ref{lemmaidreal} $\gta$ admits a normal primary decomposition $\gta=\bigcap_i\gtq_i$, such that all ideals $\gtq_i$ are saturated real-analytic prime ideals. As $\ceros(\gta)=\bigcup_i\ceros(\gtq_i)$ is an $H^{\mathsf a}$-set, we deduce by Remark \ref{rrr}(i) that each $\ceros(\gtq_i)$ is an $H^{\mathsf a}$-set. By Step 1 the equality $\ideal(\ceros(\gtq_i))=\gtq_i$ holds for each $i$. Thus,
$$
\ideal(\ceros(\gta))=\Jhaz\Big(\bigcup_{i\in I}\ceros(\gtq_i)\Big)=\bigcap_{i\in I}\ideal(\ceros(\gtq_i))=\bigcap_{i\in I}\gtq_i=\gta.
$$

\noindent{\sc Step 3}. Next we approach the general case, that is, $\gta$ is an ideal of $\an(X)$ whose zero-set is an $H^{\mathsf a}$-set. Since $\ideal(\ceros(\gta))=\ideal(\ceros(\widetilde{\sqrt[\mathsf{ra}]{\gta}}))$, it is enough to check, in view of the previous case, that $\widetilde{\sqrt[\mathsf{ra}]{\gta}}$ is a real-analytic ideal. Indeed, let $\sum_{k\geq1}a_k^2\in\widetilde{\sqrt[\mathsf{ra}]{\gta}}$ and $K\subset X$ be a compact set. By Lemma \ref{csaturation} there exists $h\in\an(X)$ such that $\ceros(h)\cap K=\varnothing$ and $h\sum_{k\geq1}a_k^2\in\sqrt[\mathsf{ra}]{\gta}$; hence, $\sum_{k\geq1}(ha_k)^2\in\sqrt[\mathsf{ra}]{\gta}$. As $\sqrt[\mathsf{ra}]{\gta}$ is real-analytic, we deduce that each $ha_k\in\sqrt[\mathsf{ra}]{\gta}$. This happens for all compact sets $K\subset X$, so we deduce by Lemma \ref{csaturation} that each $a_k\in\widetilde{\sqrt[\mathsf{ra}]{\gta}}$. Thus, $\widetilde{\sqrt[\mathsf{ra}]{\gta}}$ is a real-analytic ideal, as required.
\end{proof}

\begin{remarks}\label{varios}
Let $\gta\subset\an(X)$ be an ideal. Then

(i) $\gta\subset\sqrt[\mathsf{ra}]{\gta}\subset\sqrt[\text{\L}]{\gta}$. 

(ii) If $\ceros(\gta)$ is $H^{\mathsf a}$-set, we have $\widetilde{\sqrt[\mathsf{ra}]{\gta}}=\widetilde{\sqrt[\text{\L}]{\gta}}=\ideal(\ceros(\gta))$. However, we can only assure $\sqrt[\mathsf{ra}]{\gta}=\sqrt[\text{\L}]{\gta}$ if $\sqrt[\mathsf{ra}]{\gta}$ is in addition saturated.

(iii) Let $f\in\an(\R^n)$ be an analytic function that is an infinite sum of squares of meromorphic functions on $\R^n$. Then the ideal $\gta:=f\an(\R^n)$ is not real-analytic.

Indeed, by \cite[4.1]{abfr3} there exist $h_0,h_k\in\an(\R^n)$ such that $\ceros(h_0)\subset\ceros(f)$ and $h_0^2f=\sum_{k\geq 1}h_k^2$. Let $m\geq0$ be the greatest integer such that $f^m$ divides each $h_k$ for $k\geq1$. We write $h_0^2f=f^{2m}\sum_{k\geq 1}{h_k'}^2$ for some $h_k'\in\an(\R^n)$; hence, $f^m$ divides $h_0$ and we have $h_0'^2f^{2m+1}=f^{2m}\sum_{k\geq 1}{h_k'}^2$ for some $h_0'\in\an(\R^n)$. When simplifying, we obtain ${h_0'}^2f=\sum_{k\geq 1}{h_k'}^2$. Assume by contradiction that $\gta$ is real-analytic. Then $f$ divides $h_k'$ for all $k\geq1$, which is a contradiction.
\end{remarks}

\subsection{Quasi-real ideals} 
We introduce the next concepts to relate the real and the real-analytic radicals with the classical one. We saw that each convex ideal $\gta$ verifies $\sqrt{\gta}={\sqrt[\textsf{\L}]{\gta}}$. The type of ideals that play a similar role with respect to the real radical are  defined as follows \cite{al,gt,bp}.

\begin{deflem}\label{2}
\em Let $(X,\an_X)$ be a real coherent reduced analytic space and $\gta$ an ideal of $\an(X)$. We define the \em square root of $\gta$ \em by 
$$
\sqrt[2]{\gta}:=\Big\{f\in\an(X):\ \exists\,a_i\in\an(X)\text{ such that }f^2+\sum_{k\geq1}a_k^2\in\gta\Big\}.
$$\em
Then $\sqrt[2]{\gta}$ is an ideal, $\gta\subset\sqrt[2]{\gta}\subset\sqrt[\mathsf{ra}]{\gta}$ and $\sqrt[\mathsf{ra}]{\gta}=\sqrt{\sqrt[2]{\gta}}=\bigcup_{k\geq1}\sqrt[2^k]{\gta}$ where $\sqrt[2^k]{\gta}:=\sqrt[2]{\sqrt[2^{k-1}]{\gta}}$ for $k\geq2$. Moreover, $\gta$ is a real-analytic ideal if and only if $\gta=\sqrt[2]{\gta}$.
\end{deflem}
\begin{proof}
The only non-trivial point to prove that $\sqrt[2]{\gta}$ is an ideal is to check that it is closed under addition. This follows from the following classical trick that we recall here for the sake of completeness. Indeed, suppose that $f^2+\sum_{k\geq1}a_k^2,g^2+\sum_{k\geq1}b_k^2\in\gta$. Thus,
$$
(f+g)^2+(f-g)^2+2\Big(\sum_{k\geq1}a_k^2+\sum_{k\geq1}b_k^2\Big)=2\Big(f^2+g^2+\sum_{k\geq1}a_k^2+\sum_{k\geq1}b_k^2\Big) \in\gta,
$$
so $f+g\in\sqrt[2]{\gta}$.

To prove the equality $\sqrt[\mathsf{ra}]{\gta}=\bigcup_{k\geq1}\sqrt[2^k]{\gta}$, it is enough to show $\sqrt[\mathsf{ra}]{\gta}\subset\bigcup_{k\geq1}\sqrt[2^k]{\gta}$. Indeed, if $f\in\sqrt[\mathsf{ra}]{\gta}$, there exist $m\geq1$ and $a_k\in\an(X)$ such that $f^{2m}+\sum_{k\geq1}a_k^2\in\gta$. We may assume $2m=2^r$, so $f^{2^r}+\sum_{k\geq1}a_k^2\in\gta$. Thus, $f^{2^{r-1}}\in\sqrt[2]{\gta}$; hence, $f\in\sqrt[2^r]{\gta}$.

We show next that $\gta$ is radical if $\gta=\sqrt[2]{\gta}$. Indeed, if $f^m\in\gta$, we may assume $m=2^r$. Thus, $f^{2^{r-1}}\in\sqrt[2]{\gta}=\gta$ and proceeding inductively, we deduce $f\in\gta$. 

Consequently, if $\gta=\sqrt[2]{\gta}$, then $\sqrt[\mathsf{ra}]{\gta}=\sqrt{\sqrt[2]{\gta}}=\sqrt{\gta}=\gta$. The converse is immediate.
\end{proof}

We explore now the relations between the convex hull and the square root of an ideal $\gta$ of $\an(X)$ whose zero-set $\ceros(\gta)$ is an $H^{\mathsf a}$-set (analogous statements hold when $\ceros(\gta)$ is an $H$-set). Consider the ideal
\begin{equation}\label{ir2}
{\mathfrak r}_2(\gta):=\{g\in\an(X):\ \exists\,b\in\an(X)\text{ such that }\ceros(b)\subset\ceros(g)\ \text{and}\ bg\in\sqrt[2]{\gta}\}.
\end{equation}
 
\begin{remarks}\label{2r}
Let $\gta$ be an ideal of $\an(X)$ whose zero-set $\ceros(\gta)$ is an $H^{\mathsf a}$-set. Then

(i) In view of Theorem \ref{h17nss} and Lemma \ref{2} we have $\ideal(\ceros(\gta))=\widetilde{\sqrt[\mathsf{ra}]{\gta}}=\widetilde{\sqrt{\sqrt[2]{\gta}}}$.

(ii) Moreover, $(\sqrt[2]{\gta})^2\subset\gtg(\gta)\subset{\mathfrak r}_2(\tilde{\gta})$. 

For the first inclusion pick $f,g\in\sqrt[2]{\gta}$. We have to show $fg\in\gtg(\gta)$. Observe $fg=\frac{1}{2}((f+g)^2-f^2-g^2)$. Thus, it is enough to prove that if $f\in\sqrt[2]{\gta}$, then $f^2\in\gtg(\gta)$. Indeed, if $f\in\sqrt[2]{\gta}$, there exists $a_k\in\an(X)$ such that $f^2\leq f^2+\sum_{k\geq1}a_k^2\in\gta$. Thus, $f^2\in\gtg(\gta)$.

For the second inclusion we proceed as follows. Let $g\in\gtg(\gta)$. By Lemma \ref{crespina} there exist a non-negative $f\in\tilde{\gta}$ and $m\geq1$ such that $\ceros(f)=\ceros(\gta)$ and $f-g^2\geq0$. Observe $\ceros(f)\subset\ceros(g)$ and taking $f':=2f\in\gta$ instead of $f$, we may assume $\ceros(f)=\ceros(f-g^2)$. Indeed,
$$
\ceros(f'-g^2)=\ceros(f+(f-g^2))=\ceros(f)\cap\ceros(f-g^2)=\ceros(f)\cap\ceros(g^2)=\ceros(f').
$$
Now, since $\ceros(\gta)$ is an $H^{\mathsf a}$-set, we deduce by \cite[4.1]{abfr3} that there exist $m\geq1$ and $b,a_k\in\an(X)$ such that $\ceros(b)\subset\ceros(f-g^2)=\ceros(f)$ and $b^2(f-g^2)=\sum_{k\geq1}a_k^2$. Thus, $(bg)^2+\sum_{k\geq1}a_k^2=b^2f\in\tilde{\gta}$, so $bg\in\sqrt[2]{\tilde{\gta}}$, that is, $g\in{\mathfrak r}_2(\tilde{\gta})$.

(iii) By Theorem \ref{perfect2} and the previous remark $\sqrt[\text{\L}]{\gta}=\sqrt{\gtg(\gta)}\subset\sqrt{{\mathfrak r}_2(\tilde{\gta})}\subset\widetilde{\sqrt[\text{\L}]{\gta}}$.
\end{remarks}

We present some relations between the square root and the convex hull of an ideal.

\begin{lem}
Let $X\subset\R^n$ be a $C$-analytic set and $\gta$ an ideal of $\an(X)$ whose zero-set $\ceros(\gta)$ is an $H^{\mathsf a}$-set. Then
\begin{itemize}
\item[(i)] If $\gta$ is convex, then $\sqrt[2]{\gta}$ is also convex.
\item[(ii)] $(\widetilde{\sqrt[2]{\gta}})^4\subset(\widetilde{\gtg(\gta)})^2\subset\gtg(\gta)\subset{\mathfrak r}_2(\tilde{\gta})$.
\end{itemize}
\end{lem}
\begin{proof}
(i) Let $g\in\an(X)$ and $f\in\sqrt[2]{\gta}$ be such that $|g|\leq f$; hence $g^2\leq f^2$. As $f\in\sqrt[2]{\gta}$, there exist $a_k\in\an(X)$ such that $f^2+\sum_{k\geq1}a_k^2\in\gta$. Since
$$
g^2+\sum_{k\geq1}a_k^2\leq f^2+\sum_{k\geq1}a_k^2
$$ 
and $\gta$ is convex, we deduce $g^2+\sum_{k\geq1}a_k^2\in\gta$; hence, $g\in\sqrt[2]{\gta}$.

(ii) follows straightforwardly from Lemma \ref{cs}, Remarks \ref{2r} and the fact that $(\widetilde{\gtb})^2\subset\widetilde{\gtb^2}$ for each ideal $\gtb$ of $\an(X)$.
\end{proof}

One can unify the notions of convex hull $\gtg(\gta)$ and square root $\sqrt[2]{\gta}$ of an ideal $\gta$ of $\an(X)$ under the following general concept. A similar definition concerning {\em defining ideals} appears in \cite{gt}. 

\begin{defn}
Let $(X,\an_X)$ be a real coherent reduced analytic space. We say that an ideal $\gta$ of $\an(X)$ is \em quasi-real \em if its radical $\sqrt{\gta}$ is a real-analytic ideal. 
\end{defn}

\begin{cor}\label{qrr}
Let $X\subset\R^n$ be a $C$-analytic set and $\gta$ a quasi-real saturated ideal of $\an(X)$. Let $\gta=\bigcap_{i\in I}\gtq_i$ be a normal primary decomposition of $\gta$ and $J$ the collection of indices corresponding to the isolated primary components of $\gta$. Then
\begin{itemize}
\item[(i)] If $i_0\in J$, $\sqrt{\gtq_{i_0}}$ is a convex saturated prime ideal.
\item[(i)] If $\ceros(\gta)$ is an $H^{\mathsf a}$-set, $\widetilde{\sqrt[\mathsf{ra}]{\gta}}=\bigcap_{j\in J}\sqrt{\gtq_j}=\widetilde{\sqrt{\gta}}$.
\end{itemize}
\end{cor}
\begin{proof}
(i) Let $h_{i_0}\in\bigcap_{j\neq i_0}\gtq_j\setminus\sqrt{\gtq_{i_0}}$ (see the proof of Corollary \ref{ahora}). We have to prove that $\sqrt{\gtq_{i_0}}$ is a real-analytic ideal. Let $a_k\in\an(\R^n)$ be such that $f=\sum_{k\geq1}a_k^2\in\sqrt{\gtq_{i_0}}$. Then there exists $m\geq1$ such that $f^m\in\gtq_{i_0}$. Consequently there exists for each $k$  a sum of squares $\sigma_k$ in $\an(\R^m)$ such that $a_k^{2m}+\sigma_k\in\gtq_{i_0}$; hence, $h_{i_0}^2a_k^{2m}+h_{i_0}^2\sigma_k\in\gta$. As $\gta$ is quasi-radical, $h_{i_0}a_k^m\in\gtq_{i_0}$. Since $h_{i_0}\not\in\sqrt{\gtq_{i_0}}$, there exists $\ell\geq1$ such that $a_k^{m\ell}\in\gtq_{i_0}$, so $a_k\in\sqrt{\gtq_{i_0}}$. Thus, $\sqrt{\gtq_{i_0}}$ is a real ideal, so $\gtq_{i_0}$ is quasi-radical.

(ii) By Theorem \ref{h17nss} and Remark \ref{int} we conclude
$$
\widetilde{\sqrt[\mathsf{ra}]{\gta}}=\ideal(\ceros(\gta))=\bigcap_{j\in J}\ideal(\ceros(\gtq_j))=\bigcap_{j\in J}\sqrt{\gtq_j}=\bigcap_{i\in I}\sqrt{\gtq_i}=\widetilde{\sqrt{\gta}},
$$
as required.
\end{proof}

\begin{remarks}\label{qrr0}
(i) Under the hypotheses of Corollary \ref{qrr} the corresponding result is no longer true if $\sqrt{\gtq_{i_0}}$ is the radical of an immersed radical component of $\gta$. Use Example \ref{ahora2}(ii).

(ii) If we work in the framework of quasi-real saturated ideals, it holds an analogous result to Theorem \ref{fo} when substituting `Stein space' by `$C$-analytic set' and `closed ideal' by `quasi-real saturated ideal of $\an(X)$ whose zero-set is either an $H$-set or an $H^{\mathsf a}$-set'. The proof runs analogously to the one of Theorem \ref{fo} (\cite[\S5.Satz 9]{of}). 
\end{remarks}

\section{Real Nullstellens\"atze and complex analytic germs at $\R^n$}\label{s7}

\subsection{Saturated primary ideals and complex analytic germs at $\R^n$}
The results we present for $X=\R^n$ can be extended to an arbitrary $C$-analytic set via the correspondence theorem for ideals.

\begin{defn}
Let $\gta\subset\an(\R^n)$ be a saturated ideal. We extend the coherent sheaf $\gta\an_X$ to a coherent sheaf of ideals $\Fhaz$ on an invariant open Stein neighbourhood $\Omega$ of $\R^n$ in $\C^n$. The analytic germ $Y_{\R^n}$ at $\R^n$ of the support $Y:=\supp({\mathcal F})$ will be called the \em complex zero-set $\ceros_{\C}(\gta)$ of $\gta$\em. 
\end{defn}

\begin{lem}\label{pasito}
Let $\gtq\subset\an(\R^n)$ be a primary saturated ideal. Then $f\in\gtp:=\sqrt{\gtq}$ if and only if there exists an open neighbourhood $\Omega$ of $\R^n$ in $\C^n$, a holomorphic extension $F$ of $f$ to $\Omega$ and a representative $Y$ of $\ceros_{\C}(\gtq)$ in $\Omega$ such that $Y\subset\ceros(F)$. In other words, $f\in\gtp:=\sqrt{\gtq}$ if and only if $F$ vanishes identically on $\ceros_{\C}(\gtq)$.
\end{lem}
\begin{proof}
The `only if' implication follows from the following facts:
\begin{itemize}
\item[(1)] if $\gtq$ is saturated, then $\gtp$ is also saturated and 
\item[(2)] $f\in\ideal(\ceros(\gtq))$ implies that $F$ vanishes identically on $\ceros_{\C}(\gtq)$. 
\end{itemize}
For the `if' implication let $Y$ be a representant of $\ceros_{\C}(\gtq)$ on a suitable complex neighbourhood {\color{red}of} $\R^n$ in $\C^n$ such that $Y\subset \ceros(F)$. Pick a point
$$
x\in\ceros(\gtq)=\ceros_{\C}(\gtq)\cap\R^n=Y\cap\R^n\subset\ceros(F)\cap\R^n=\ceros(f).
$$
Since $Y\subset\ceros(F)$, we have $F_x\in\ideal(\ceros(Y_x))=\ideal(\ceros(\gtq_x\an_{\C^n,x}))=\sqrt{\gtq_x\an_{\C^n,x}}$. Thus, there exists $m\geq1$ such that $F^m_x\in\gtq_x\an_{\C^n,x}$. By Lemma \ref{lemmaprimary} we have $f^m=(F|_{\R^n})^m\in\gtq$, so $f\in\gtp$.
\end{proof}
\begin{remarks}\label{cdiff}
(i) Let $\gta_1,\gta_2$ be two saturated ideals of $\an(X)$ such that $\gta_1\subset\gta_2$. Then 
$$
\ceros_\C(\gta_2)\subset\ceros_\C(\gta_1).
$$

(ii) Let $\gtq_1,\gtq_2$ be two saturated primary ideals of $\an(X)$ such that $\ceros_\C(\gtq_2)\subset\ceros_\C(\gtq_1)$. Then 
$$
\sqrt{\gtq_1}\subset\sqrt{\gtq_2}.
$$
\end{remarks}

\begin{lem}\label{previo0}
Let $\gtq\subset\an(\R^n)$ be a primary saturated ideal. Then there exists an irreducible analytic germ $Z_{\R^n}$ such that $\ceros_{\C}(\gtq)=Z_{\R^n}\cup\sigma(Z_{\R^n})$. In particular, if $\ceros_{\C}(\gtq)$ is invariant, then it is also irreducible.
\end{lem}
\begin{proof}
We extend the coherent sheaf $\gtq\an_X$ to a coherent sheaf of ideals $\Fhaz$ on a contractible invariant open Stein neighbourhood $\Omega$ of $\R^n$ in $\C^n$ and denote $Y:=\supp(\Fhaz)$. Recall $\ceros_\C(\gtq)=Y_{\R^n}$. Consider the subring $\An(\Omega)$ of $H^0(\Omega,\an_{\C^n})$ of all invariant holomorphic functions on $\Omega$. Observe that the restriction homomorphism $\varphi:\An(\Omega)\to\an(\R^n),\ F\mapsto F|_{\R^n}$ is injective. Since $\gtq$ is a primary, $\gtp:=\sqrt{\gtq}$ is prime, so $\gtP:=\varphi^{-1}(\gtp)$ is also prime.

As $\ceros(\gtp)=\ceros(\gtq)\neq\varnothing$, it holds $\ceros(\gtP)\neq\varnothing$. By Cartan's Theorem $A$ and using that $\gtq$ is saturated, we deduce after shrinking $\Omega$ that $Y_{\R^n}=\ceros(\gtP)_{\R^n}$ and $Y=\ceros(\gtP)$.

Let $Y_{\R^n}=\bigcup_{i\in I}Z_{i,\R^n}$ be the decomposition of $Y_{\R^n}$ as the union of its irreducible components. Pick one of them and for simplicity denote it with $Z_{\R^n}$. By \cite[Cor.2, pag.151]{wb} (and its proof) we may assume that there exists an irreducible analytic set $Z$ in $\Omega$ whose germ in $\R^n$ is precisely $Z_{\R^n}$. Notice that $Z$ and $\sigma(Z)$ are (eventually equal) irreducible components of $Y$ because $Z_{\R^n}$ is an irreducible component of the invariant germ $Y_{\R^n}$. Assume $Y\neq Z\cup\sigma(Z)$ and let $T$ be the union of all other irreducible components of $Y$. Clearly, $T$ is invariant. Choose now invariant $F,G\in H^0(\Omega,\an_{\C^n})$ such that 
\begin{itemize}
\item $Z\cup\sigma(Z)\subset\ceros(F)$ but $T\not\subset\ceros(F)$,
\item $T\subset\ceros(G)$ but $Z\cup\sigma(Z)\not\subset\ceros(G)$.
\end{itemize}
Therefore the invariant holomorphic function $FG$ vanishes on $Y$.

Let $x\in\ceros(\gtp)=Y\cap\R^n$ and observe that we obtain by the complex local analytic Nullstellensatz
$$
\ideal(Y_x)=\ideal(\ceros(\Fhaz_x))=\ideal(\ceros(\gtq_x\an_{\C^n,x}))=\sqrt{\gtq_x\an_{\C^n,x}}.
$$
Thus, there exists $m\geq1$ such that $(FG)^m_x\in\gtq_x\an_{\C^n,x}$. By Lemma \ref{lemmaprimary} $(FG)^m\in\gtq$, so $FG\in\gtp\cap\An(\Omega)=\gtP$. As $\gtP$ is prime, we may assume $F\in\gtP$; hence, $T\subset Y=\ceros(\gtP)\subset\ceros(F)$, which is a contradiction. Consequently $Y=Z\cup\sigma(Z)$, so $Y_{\R^n}=Z_{\R^n}\cup\sigma(Z_{\R^n})$, as required. 
\end{proof}

\begin{lem}\label{irredcomp}
Let $\gta\subset\an(\R^n)$ be a saturated ideal, $\gta=\bigcap_{i\in I}\gtq_i$ a normal primary decomposition of $\gta$ and $J\subset I$ the collection of indices corresponding to the isolated primary components of $\gta$. Then $\ceros_\C(\gta)=\bigcup_{j\in J}\ceros_\C(\gtq_j)$ and for each $j\in J$ there exists an irreducible component $Z_{j,\R^n}$ of $\ceros_\C(\gta)$ such that $\ceros_\C(\gtq_j)=Z_{j,\R^n}\cup\sigma(Z_{j,\R^n})$.
\end{lem}
\begin{proof}
Observe first
$$
\ceros_\C(\gta)\cap\R^n=\ceros(\gta)=\bigcup_{i\in I}\ceros(\gtq_i)=\bigcup_{j\in J}\ceros(\gtq_j)=\bigcup_{j\in J}\ceros_\C(\gtq_j)\cap\R^n.
$$
Fix $x\in Z\cap\R^n$ and observe $\ceros_\C(\gta)_x=\ceros(\gta\an_{\C^n,x})$. Let $\gtq_{i_1},\ldots,\gtq_{i_r}$ be the primary ideals of our normal primary decomposition whose zero-sets contain $x$. We may assume that $\gtq_{i_1},\ldots,\gtq_{i_s}$ are those primary ideals among $\gtq_{i_1},\ldots,\gtq_{i_r}$, which are in addition isolated. Observe 
\begin{multline*}
\ceros_\C(\gta)_x=\ceros(\gta\an_{\C^n,x})=\ceros\Big(\bigcap_{\ell=1}^r\gtq_{i_\ell}\an_{\C^n,x}\Big)=\bigcup_{\ell=1}^r\ceros(\sqrt{\gtq_{i_\ell}}\an_{\C^n,x})\\
=\bigcup_{\ell=1}^s\ceros(\sqrt{\gtq_{i_\ell}}\an_{\C^n,x})=\bigcup_{\ell=1}^s\ceros(\gtq_{i_\ell}\an_{\C^n,x})=\bigcup_{\ell=1}^s\ceros_\C(\gtq_{i_\ell})_x=\bigcup_{j\in J}\ceros_\C(\gtq_j)_x;
\end{multline*}
hence, $\ceros_\C(\gta)=\bigcup_{j\in J}\ceros_\C(\gtq_j)$. 

For each $\ceros_\C(\gtq_j)$ there exists by Lemma \ref{previo0} an irreducible analytic germ $Z_{j,\R^n}$ {\color{red}at} $\R^n$ such that $\ceros_\C(\gtq_j)=Z_{j,\R^n}\cup\sigma(Z_{j,\R^n})$; hence, $\ceros_\C(\gta)=\bigcup_{j\in J}Z_{j,\R^n}\cup\sigma(Z_{j,\R^n})$. By Remark \ref{cdiff} and the fact that the primary ideals $\gtq_j$ are isolated, we deduce $Z_{j,\R^n}\not\subset Z_{j',\R^n}\cup\sigma(Z_{j',\R^n})$ if $j\neq j'$. Thus, for each $j\in J$ the germs $Z_{j,\R^n}$ and $\sigma(Z_{j,\R^n})$ are irreducible components of $\ceros_\C(\gta)$.
\end{proof}
Now we are ready to prove Theorem \ref{complexreal}.

\begin{proof}[Proof of Theorem \em \ref{complexreal}]
By Lemma \ref{previo0} there exists an irreducible analytic germ $Z_{\R^n}$ such that $\ceros_\C(\gtq)=Z_{\R^n}\cup\sigma(Z_{\R^n})$. Now we prove the following implications.

(i) $\Longrightarrow$ (ii) As $\ideal(\ceros(\gtq))=\sqrt{\gtq}$, we deduce by \cite[pag.154]{wb} that $\ceros_\C(\sqrt{\gtq})=\ceros_\C(\gtq)$ is the germ of the `complexification' of $\ceros(\sqrt{\gtq})=\ceros(\gtq)$ at $\R^n$. Since the dimension of the `complexification' of $\ceros(\gtq)$ coincides with its dimension \cite[\S8. Prop.12]{wb}, we deduce $\dim(\ceros_\C(\gtq))=\dim(\ceros(\gtq))$.

(ii) $\Longrightarrow$ (i) Let $Y_{\R^n}$ be the germ of the `complexification' of $\ceros(\gtq)$ at $\R^n$. By \cite[pag.154]{wb} we have $Y_{\R^n}\subset Z_{\R^n}\cap\sigma(Z_{\R^n})$. Since $Z_{\R^n}$ is irreducible, we get that either $Z_{\R^n}=\sigma(Z_{\R^n})$ or $\dim(Z_{\R^n}\cap\sigma(Z_{\R^n}))<\dim(Z_{\R^n})$. But this is impossible because then
\begin{multline*}
\dim(\ceros(\gtq))=\dim(Y_{\R^n})\leq\dim(Z_{\R^n}\cap\sigma(Z_{\R^n}))<\dim(Z_{\R^n})\\
\leq\dim(Z_{\R^n}\cap\sigma(Z_{\R^n}))=\dim(\ceros_\C(\gtq))=\dim(\ceros(\gtq)),
\end{multline*}
which is a contradiction. Thus, $\ceros_\C(\gtq)=Z_{\R^n}$ and
$$
\dim(\ceros(\gtq))=\dim(Y_{\R^n})\leq\dim(Z_{\R^n})=\dim(\ceros_\C(\gtq))=\dim(\ceros(\gtq));
$$
hence, $\dim(Y_{\R^n})=\dim(Z_{\R^n})$ and as $Z_{\R^n}$ is irreducible, $Y_{\R^n}=Z_{\R^n}$. Thus, by Lemma \ref{pasito} we have $f\in\sqrt{\gtq}$ if and only if there exists an open neighbourhood $\Omega$ {\color{red}of} $\R^n$ in $\C^n$, a holomorphic extension $F$ of $f$ to $\Omega$ and a complex analytic subset $T\subset\ceros(F)$ in $\Omega$ such that $T_{\R^n}=\ceros_{\C}(\gtq)=Z_{\R^n}$. 

On the other hand, by \cite[pag.154]{wb} we have that $g\in\ideal(\ceros(\gtq))$ if and only if there exists an open neighbourhood $\Omega$ of $\R^n$ in $\C^n$, a holomorphic extension $G$ of $g$ to $\Omega$ and a complex analytic subset $S\subset\ceros(F)$ in $\Omega$ such that $S_{\R^n}=Y_{\R^n}$. 

We conclude $\ideal(\ceros(\gtq))=\sqrt{\gtq}$ because $Z_{\R^n}=Y_{\R^n}$.

(ii) $\Longrightarrow$ (iii) is straightforward.

(iii) $\Longrightarrow$ (ii) Let $(\Omega,Z)$ be such that $\Omega$ is an open invariant neighbourhood of $\R^n$ in $\C^n$ and $Z$ is an irreducible representative of $Z_{\R^n}$ in $\Omega$ (see \cite[Cor.2, pag.151]{wb}). The irreducibility of $Z$ guarantees that it is pure dimensional; hence, so is $Z\cup\sigma(Z)$. We have
\begin{multline*}
\dim(\ceros(\gtq))\geq\dim(\ceros(\gtq)_x)=\dim(\ceros(\gtq\an_{\R^n,x}))=\dim(\ceros(\gtq\an_{\C^n,x}))\\
=\dim(Z_x\cup\sigma(Z)_x)=\dim(Z_{\R^n}\cup\sigma(Z_{\R^n}))=\dim(\ceros_\C(\gtq))\geq\dim(\ceros(\gtq);
\end{multline*}
hence, $\dim(\ceros_\C(\gtq))=\dim(\ceros(\gtq))$, as required. 
\end{proof}

We finish this section with some arithmetic considerations.

\begin{defn}
A finite set ${\mathfrak F}:=\{f_1,\ldots,f_m\}\subset\an(\R^n)$ is \em sharp \em if $\dim(\ceros_\C(f_1,\ldots,f_m))=n-m$. 
\end{defn}
\begin{remarks}\label{bd}
(i) Let ${\mathfrak F}:=\{f_1,\ldots,f_m\}\subset\an(\R^n)$. For each $\ell=1,\ldots,m$ the (finitely generated) ideal $\gtb_\ell=(f_1,\ldots,f_\ell)\an(X)$ is saturated, so it admits a normal primary decomposition $\gtb_\ell=\bigcap_{j\in J_\ell}\gtq_{j\ell}$. Then it holds: \em ${\mathfrak F}$ is a sharp family if and only if $f_\ell$ does not belong to any of the minimal prime ideals of the family $\{\sqrt{\gtq_{j,\ell-1}}\}_{j\in J_{\ell-1}}$ for each $\ell=2,\ldots,m$\em. 

Let $\Omega$ be an open neighbourhood of $\R^n$ in $\C^n$, on which each $f_i$ admits a holomorphic extension $F_i$. Recall the following well-known consequence of the Identity Principle:

\begin{substeps}{bd}\label{bd0}
\em If $Y$ is an irreducible complex analytic subset of $\Omega$, then $Y$ is pure dimensional and if $F\in H^0(\Omega,\an_{\C^n})$, then either $Y\subset\ceros(F)$ or $\dim(Y\cap\ceros(F))<\dim(Y)$. 
\end{substeps}

Thus, shrinking the open set $\Omega$ in each step, it follows from \ref{bd}.\ref{bd0} that $\dim(\ceros_\C(f_1,\ldots,f_m))\geq n-m$. By Lemmas \ref{previo0}, \ref{irredcomp} and \ref{bd}.\ref{bd0} it holds $\dim(\ceros_\C(f_1,\ldots,f_m))=n-m$ if and only if $f_\ell$ does not belong to any of the minimal prime ideals of the family $\{\sqrt{\gtq_{j,\ell-1}}\}_{j\in J_{\ell-1}}$ for each $\ell=2,\ldots,m$. As this kind of argument is standard, we leave the concrete details to the reader \cite[footnote 9, p. 96-97]{c2}.

(ii) If $\gta\subset\an(\R^n)$ is an ideal, we have 
$$
\sup\{{\rm card}({\mathfrak F}):\ {\mathfrak F}\subset\gta\text{ is sharp}\}=n-\dim(\ceros_\C(\gta))\leq n-\dim(\ceros(\gta)).
$$
If $\gtq$ is a primary ideal of $\an(\R^n)$, we obtain by Theorem \ref{complexreal}
$$
\ideal(\ceros(\gtq))=\sqrt{\gtq}\quad\text{ if and only if}\quad\sup\{{\rm card}({\mathfrak F}):\ {\mathfrak F}\subset\gtq\text{ is sharp}\}=n-\dim(\ceros(\gtq)).
$$
\end{remarks}

\begin{lem}\label{prin}
Let $\gtq\subset\an(\R^n)$ be a primary saturated ideal. Then $\gtq$ is 
a principal ideal if and only if $\sqrt{\gtq}$ is a principal ideal.
\end{lem}
\begin{proof}
For the `if' implication, assume that $\sqrt{\gtq}$ is a principal ideal generated by $f\in\an(\R^n)$. One can check that $\gtq$ is generated by $f^k$ where $k=\min\{m\geq1:\ f^m \in\gtq\}$. 

Conversely, assume that $\gtq$ is generated by $f\in\an(\R^n)$. By \cite[Prop.3]{cain} there exists $h\in\an(\R^n)$ such that $h_x\an_x=\sqrt{f_x\an_x}$ for each point $x \in\R^n$. We claim $\sqrt{\gtq}=h\an(\R^n)$. 

Indeed, if $g\in\sqrt{\gtq}=\sqrt{f\an(\R^n)}$, the germ $g_x\in\sqrt{f_x\an_x}=h_x\an_x$ for each $x\in\R^n$, so $g\in h\an(\R^n)$. Now we prove $h\in\sqrt{\gtq}$. Pick a point $x\in\ceros(\gtq)$. As $h_x\an_x=\sqrt{f_x\an_x}$, we find an integer $m$ such that $h_x^m\in f_x\an_x=\gtq\an_x$. Since $\gtq$ is a saturated primary ideal, Lemma \ref{lemmaprimary} implies $h^m\in\gtq$, as required.
\end{proof}

\begin{cor}
Let $\gtq\subset\an(\R^n)$ be a primary saturated ideal. We have
\begin{itemize}
\item[(i)] If $\dim(\ceros(\gtq))=n-1$, then $\ideal(\ceros(\gtq))=\sqrt{\gtq}$.
\item[(ii)] If $\dim(\ceros(\gtq))=n-2$, then $\ideal(\ceros(\gtq))=\sqrt{\gtq}$ if and only if $\gtq$ is not principal.
\end{itemize}
\end{cor}
\begin{proof}
(i) follows from Theorem \ref{complexreal} because
$$
n-1=\dim(\ceros(\gtq))\leq\dim(\ceros_\C(\gtq))\leq n-1.
$$

(ii) Assume first $\ideal(\ceros(\gtq))=\sqrt{\gtq}$. Then
$$
n-2=\dim(\ceros(\gtq))=\dim(\ceros_\C(\gtq))=n-\sup\{{\rm card}({\mathfrak F}):\ {\mathfrak F}\subset\gta\text{ is sharp}\};
$$
hence, $\gtq$ is not principal. Conversely, if $\gtq$ is not principal, then
$$
n-2=\dim(\ceros(\gtq))\leq\dim(\ceros_\C(\gtq))\leq n-2
$$
and by Theorem \ref{complexreal} we obtain $\ideal(\ceros(\gtq))=\sqrt{\gtq}$.
\end{proof}

\renewcommand\refname{References}

\end{document}